\documentclass{article}
\usepackage{amsmath,amssymb,amsthm,graphicx,epsfig,subfigure,float,url}
\usepackage[colorlinks=true]{hyperref}
\usepackage{pdfsync}

\newcommand{\R}{\mathrm{I\kern-0.21emR}}
\newcommand{\N}{\mathrm{I\kern-0.21emN}}

\renewcommand{\leq}{\leqslant}

\newtheorem{theorem}{Theorem}

\newtheorem{lemma}{Lemma}
\theoremstyle{definition}\newtheorem{example}{Example}
\theoremstyle{definition}\newtheorem{remark}{Remark}

\title{The turnpike property in finite-dimensional nonlinear optimal control}

\author{Emmanuel Tr\'elat\footnote{Sorbonne Universit\'es, UPMC Univ Paris 06, CNRS UMR 7598, Laboratoire Jacques-Louis Lions, Institut Universitaire de France, F-75005, Paris, France \texttt{emmanuel.trelat@upmc.fr}.}
\and Enrique Zuazua\footnote{BCAM - Basque Center for Applied Mathematics, Mazarredo, 14 E-48009 Bilbao-Basque Country-Spain.}\ \footnote{Ikerbasque, Basque Foundation for Science, Alameda Urquijo 36-5, Plaza Bizkaia, 48011, Bilbao-Basque Country-Spain (\texttt{zuazua@bcamath.org}).}
}

\date{}

\begin{document}

\maketitle

\begin{abstract}
Turnpike properties have been established long time ago in finite-dimensional optimal control problems arising in econometry. They refer to the fact that, under quite general assumptions, the optimal solutions of a given optimal control problem settled in large time consist approximately of three pieces, the first and the last of which being transient short-time arcs, and the middle piece being a long-time arc staying exponentially close to the optimal steady-state solution of an associated static optimal control problem. We provide in this paper a general version of a turnpike theorem, valuable for nonlinear dynamics without any specific assumption, and for very general terminal conditions. Not only the optimal trajectory is shown to remain exponentially close to a steady-state, but also the corresponding adjoint vector of the Pontryagin maximum principle.
The exponential closedness is quantified with the use of appropriate normal forms of Riccati equations. We show then how the property on the adjoint vector can be adequately used in order to initialize successfully a numerical direct method, or a shooting method. In particular, we provide an appropriate variant of the usual shooting method in which we initialize the adjoint vector, not at the initial time, but at the middle of the trajectory.
\end{abstract}


\medskip


\section{Introduction and main result}\label{sec:intro}

\paragraph{Dynamical optimal control problem.}
Consider the nonlinear control system
\begin{equation}\label{contsys}
\dot{x}(t) = f(x(t),u(t)),
\end{equation}
where $f:\R^n\times\R^m\rightarrow\R^n$ is of class $C^2$.
Let $R=(R^1,\ldots,R^k):\R^n\times\R^n\rightarrow\R^k$ be a mapping of class $C^2$, and let $f^0:\R^n\times\R^m\rightarrow\R$ be a function of class $C^2$ .
For a given $T>0$ we consider the optimal control problem $\bf (OCP)_T$ of determining a control $u_T(\cdot)\in L^\infty(0,T;\R^m)$ minimizing the cost functional
\begin{equation}\label{cost}
C_T(u) = \int_0^T f^0(x(t),u(t))\,  dt
\end{equation}
over all controls $u(\cdot)\in L^\infty(0,T;\R^m)$, where $x(\cdot)$ is the solution of \eqref{contsys} corresponding to the control $u(\cdot)$ and such that
\begin{equation}\label{terminalcond}
R(x(0),x(T))=0.
\end{equation}

We assume throughout that $\bf (OCP)_T$ has an optimal solution $(x_T(\cdot),u_T(\cdot))$.
Conditions ensuring the existence of an optimal solution are well known (see, e.g., \cite{Cesari,Trelatbook}). For example, if the set of velocities $\{ f(x,u) \mid u\in\R^m \}$ is a convex subset of $\R^n$ for every $x\in\R^n$ and if the epigraph of $f^0$ is convex, then there exists at least one optimal solution. Note that this is the case whenever the system \eqref{contsys} is control-affine, that is, $f(x,u) = f_0(x) + \sum_{i=1}^m u_i f_i(x)$, where the $f_i$'s, $i=0,\ldots,m$, are $C^1$ vector fields in $\R^n$, and $f^0$ is a positive definite quadratic form in $(x,u)$. The classical linear quadratic problem fits in this class (and in that case the optimal solution is moreover unique).

According to the Pontryagin maximum principle (see \cite{AgSachkov,Pontryagin,Trelatbook}), there must exist an absolutely continuous mapping $\lambda_T(\cdot):[0,T]\rightarrow\R^n$, called \textit{adjoint vector}, and a real number $\lambda^0_T\leq 0$, with $(\lambda_T(\cdot),\lambda^0_T)\neq (0,0)$, 
such that, for almost every $t\in[0,T]$,
\begin{equation}\label{extremalsyst}
\begin{split}
&\dot{x}_T(t) = \frac{\partial H}{\partial \lambda}(x_T(t),\lambda_T(t),\lambda^0_T, u_T(t)), \\
&\dot{\lambda}_T(t) = -\frac{\partial H}{\partial x}(x_T(t),\lambda_T(t),\lambda^0_T, u_T(t)),\\
&\frac{\partial H}{\partial u} (x_T(t),\lambda_T(t),\lambda^0_T, u_T(t)) = 0,
\end{split}
\end{equation}
where the Hamiltonian $H$ of the optimal control problem is defined by
\begin{equation}\label{defH}
H(x,\lambda,\lambda^0,u) = \langle\lambda,f(x,u)\rangle +\lambda^0 f^0(x,u),
\end{equation}
for all 
$(x,\lambda,u)\in\R^n\times\R^n\times\R^m$. Moreover we have transversality conditions: there exist $(\gamma_1,\ldots,\gamma_k)\in \R^k$ such that
\begin{equation}\label{condtransv}
\begin{pmatrix} -\lambda_T(0)\\ \phantom{-} \lambda_T(T)\end{pmatrix} = \sum_{i=1}^k \gamma_i \nabla R^i(x_T(0),x_T(T)).
\end{equation}

\begin{remark}\label{remBC}
The integer $k\leq 2n$ is the number of relations imposed to the terminal conditions in $\bf (OCP)_T$. Let us describe some typical situations.
\begin{itemize}
\item If the initial and final points are fixed in $\bf (OCP)_T$, that is, if we impose that $x(0)=x_0$ and $x(T)=x_1$ in the optimal control problem, then $k=2n$ and $R(x,y)=(x-x_0,y-x_1)$. The transversality condition \eqref{condtransv} gives no additional information in that case.
\item If the initial point is fixed in $\bf (OCP)_T$ and the final point is let free, then $k=n$ and $R(x,y)=x-x_0$. The transversality condition \eqref{condtransv} then implies that $\lambda_T(T)=0$.
\item If the initial point is fixed and if the final point is subject to the constraint $g(x_T(T))=0$, with $g=(g^1,\ldots,g^p):\R^n\rightarrow\R^p$, then $k=n+p$ and $R(x,y)=(x-x_0,g(y))$. The transversality condition \eqref{condtransv} then implies that $\lambda(T)$ is a linear combination of the vectors $\nabla g^i(x_T(T))$, $i=1,\ldots,k$.
\item If the periodic condition $x_T(0)=x_T(T)$ is imposed in $\bf (OCP)_T$, then $k=n$ and $R(x,y)=x-y$. In that case, the transversality condition \eqref{condtransv} yields that $\lambda_T(0)=\lambda_T(T)$.
\end{itemize}
\end{remark}

The quadruple $(x_T(\cdot),\lambda_T(\cdot),\lambda^0_T,u_T(\cdot))$ is called an \textit{extremal lift} of the optimal trajectory. The adjoint vector $(\lambda_T(T),\lambda^0_T)$ can be interpreted as a \textit{Lagrange multiplier} of the optimal control problem viewed as a constrained optimization problem (see \cite{Trelatbook}). It is defined up to a multiplicative scalar.
The extremal is said to be \textit{normal} whenever $\lambda^0_T\neq 0$, and in that case the adjoint vector is usually normalized so that $\lambda^0_T=-1$. The extremal is said to be \textit{abnormal} whenever $\lambda^0_T=0$.
Note that every extremal is normal (that is, the Lagrange multiplier associated with the cost is nonzero) if for instance $R(x,y)=x-x_0$ (that is, fixed initial point and free final point).

Throughout the paper, we assume that the optimal solution $(x_T(\cdot),u_T(\cdot))$ of $\bf (OCP)_T$ under consideration has a normal extremal lift $(x_T(\cdot),\lambda_T(\cdot),-1,u_T(\cdot))$.
As it is by now well known, such an assumption is automatically satisfied under generic assumptions (see \cite{CJT,RT}), or under controllability assumptions (see \cite{BettiolFrankowska,Vinter}). 

\paragraph{Static optimal control problem.}
Besides, we consider the \textit{static} optimal control problem
\begin{equation}\label{staticpb}
\min_{\stackrel{(x,u)\in\R^n\times\R^m}{f(x,u)=0}} f^0(x,u).
\end{equation}
This is a usual optimization problem settled in $\R^n\times\R^m$ with a nonlinear equality constraint.
Note that, as it will become clear further, this problem is only related with the dynamical part of the previous $\bf (OCP)_T$ (the terminal conditions do not enter into play here).

We assume that this minimization problem has a solution $(\bar x,\bar u)$.
Note that the minimizer exists and is unique whenever $f$ is linear in $x$ and $u$ for instance, and $f^0$ is a positive definite quadratic form in $(x,u)$.
According to the Lagrange multipliers rule, there exists $(\bar\lambda,\bar\lambda^0)\in\R^n\times\R\setminus\{(0,0)\}$, with $\bar\lambda^0\leq 0$, such that
\begin{equation*}
\begin{split}
f(\bar x,\bar u)&=0,\\
\bar\lambda^0\frac{\partial f^0}{\partial x}(\bar x,\bar \lambda,\bar u) + \Big\langle\bar\lambda,\frac{\partial f}{\partial x}(\bar x,\bar \lambda,\bar u)\Big\rangle &= 0,\\
\bar\lambda^0\frac{\partial f^0}{\partial u}(\bar x,\bar \lambda,\bar u) + \Big\langle\bar\lambda,\frac{\partial f}{\partial u}(\bar x,\bar \lambda,\bar u)\Big\rangle &= 0,
\end{split}
\end{equation*}
or in other words, using the Hamiltonian $H$ defined by \eqref{defH},
\begin{equation}\label{static_extr}
\begin{split}
\frac{\partial H}{\partial \lambda}(\bar x,\bar \lambda,\bar \lambda^0,\bar u)&=0,\\
-\frac{\partial H}{\partial x}(\bar x,\bar \lambda,\bar \lambda^0,\bar u)&=0,\\
\frac{\partial H}{\partial u}(\bar x,\bar \lambda,\bar \lambda^0,\bar u) &= 0 .
\end{split}
\end{equation}
This is the optimality system of the static optimal control problem \eqref{staticpb}.

Throughout the paper we assume that the abnormal case does not occur and hence we normalize the Lagrange multiplier so that $\bar\lambda^0=-1$. As it is well known, Mangasarian-Fromowitz constraint qualification conditions do guarantee normality (see \cite{MangasarianFromowitz}). For example this is true as soon as the set $\{(x,u)\in\R^n\times\R^m\mid f(x,u)=0\}$ is a submanifold, which is a very slight (and generic) assumption.

\paragraph{The turnpike property.}
Since $(\bar x,\bar\lambda,\bar u)$ is an equilibrium point of the extremal equations \eqref{extremalsyst}, it is natural to expect that, under appropriate assumptions (such as controllability assumptions), if $T$ is large then the optimal extremal solution $(x_T(\cdot),\lambda_T(\cdot),u_T(\cdot))$ of the optimal control problem $\bf (OCP)_T$ remains most of the time close to the static extremal point $(\bar x,\bar\lambda,\bar u)$. More precisely, it is expected that if $T$ is large then the extremal is approximately made of three pieces, where:
\begin{itemize}
\item the first piece is a short-time piece, defined on $[0,\tau]$ for some $\tau>0$, along which the extremal $(x_T(\cdot),\lambda_T(\cdot),u_T(\cdot))$ passes approximately from $(x_T(0),\lambda_T(0),u_T(0))$ to $(\bar x,\bar\lambda,\bar u)$;
\item the second piece is approximately stationary, identically equal to the steady-extremal $(\bar x,\bar\lambda,\bar u)$ over the long-time interval $[\tau,T-\tau]$;
\item the third piece is a short-time piece, defined on $[T-\tau,T]$, along which $(x_T(\cdot),\lambda_T(\cdot),u_T(\cdot))$ passes approximately from $(\bar x,\bar\lambda,\bar u)$ to $(x_T(T),\lambda_T(T),u_T(T))$.
\end{itemize}
The first and the third arcs are seen as transient.

At least for the trajectory (but not for the full extremal), this property is known in the existing literature, and in particular in econometry, as the \textit{turnpike property} (see an early result in \cite{Cass1965} for a specific optimal economic growth problem). It stipulates that the solution of an optimal control problem in large time should spend most of its time near a steady-state. In infinite horizon the solution should converge to that steady-state. In econometry such steady-states are known as \textit{Von Neumann points}. The turnpike property means then, in this context, that large time optimal trajectories are expected to converge, in some sense, to Von Neumann points.%
\footnote{As very well reported in \cite{MacKenzie1976}, it seems that the first turnpike result was discovered in \cite[Chapter 12]{DorfmanSamuelsonSolow}, in view of deriving efficient programs of capital accumulation, in the context of a Von Neumann model in which labor is treated as an intermediate product. As quoted by \cite{MacKenzie1976}, in this chapter one can find the following seminal explanation:
\begin{quote}
Thus in this unexpected way, we have found a real normative significance for steady growth -- not steady growth in general, but maximal von Neumann growth. It is, in a sense, the single most effective way for the system to grow, so that if we are planning long-run growth, no matter where we start, and where we desire to end up, it will pay in the intermediate stages to get into a growth phase of this kind. It is exactly like a turnpike paralleled by a network of minor roads. There is a fastest route between any two points; and if the origin and destination are close together and far from the turnpike, the best route may not touch the turnpike. But if origin and destination are far enough apart, it will always pay to get on to the turnpike and cover distance at the best rate of travel, even if this means adding a little mileage at either end. The best intermediate capital configuration is one which will grow most rapidly, even if it is not the desired one, it is temporarily optimal.
\end{quote}
This famous reference has given the name of \textit{turnpike}.}
Several turnpike theorems have been derived in the 60's for discrete-time optimal control problems arising in econometry (see, e.g., \cite{MacKenzie1963}). Continuous versions have been proved in \cite{Haurie} under quite restrictive assumptions on the dynamics motivated by economic growth models. All of them are established for point-to-point optimal control problems and give information on the trajectory only (but not on the adjoint vector).
We also refer the reader to \cite{CarlsonHaurieLeizarowitz_book1991} for an extensive overview of these continuous turnpike results (see also \cite{Zaslavski}).
More recently, turnpike phenomena have been also put in evidence in optimal control problems coming from biology, such as in \cite{CoronGabrielShang}, in relation with singular arcs (see also \cite{Rapaport}).
Note that, in \cite{BoscainPiccoli}, the word "turnpike" refers to the set of points where singular trajectories can stay. In dimension $2$ for the minimal time problem for a control-affine system $\dot x(t)=f_0(x(t))+u(t)f_1(x(t))$, with $u(t)\in[-1,1]$, it is the set of points where $f_1$ is parallel to the Lie bracket $[f_0,f_1]$.

As it is well known, the turnpike properties are due to the saddle point feature of the extremal equations of optimal control (see \cite{Rockafellar1973,Samuelson1972}), and more precisely to the Hamiltonian nature of the extremal equations inferred from the Pontryagin maximum. These results relate the turnpike property with the asymptotic stability properties of the solutions of the Hamiltonian extremal system, coming from the concavity-convexity of the Hamiltonian function.

It is noticeable that, although all these results use extensively this saddle point property, they do not seem aware of finer properties of the Hamiltonian matrix of the extremal system, pointed out in \cite{AndersonKokotovic,WildeKokotovic} and explained further.
In these articles, which, surprisingly enough, seem to have remained widely unacknowledged, the authors prove that the optimal trajectory is approximately made of two solutions for two infinite-horizon optimal control problems, which are pieced together and exhibit a similar transient behavior.
This turnpike property is shown in \cite{WildeKokotovic} for linear quadratic problems under the Kalman condition, extended in \cite{AndersonKokotovic} to nonlinear control-affine systems where the vector fields are assumed to be globally Lipschitz, and being referred to as  the exponential dichotomy property. In both cases the initial and final conditions for the trajectory are prescribed.
Their approach is remarkably simple and points out clearly the hyperbolicity phenomenon which is at the heart of the turnpike results. The use of Riccati-type reductions permits to quantify the saddle point property in a precise way.
Their proofs are however based on a Hamilton-Jacobi approach and, at the end of the article, the open question of extending their results to problems where the Hamilton-Jacobi theory cannot be used (that is, most of the time!) is formulated. Here, in particular, we solve this open question employing the Pontryagin maximum principle that yields a two-point boundary value problem (as we did above).

We provide hereafter a much more general version of a turnpike theorem, valuable without any specific assumption on the dynamics. We stress that we obtain an exponential closedness result to the steady-state, not only for the optimal trajectory in large time, but also for the control and for the associated adjoint vector. The latter property is particularly important in view of the practical implementation of a shooting method, as explained further.

\paragraph{Preliminaries and notations.}
Our analysis will consist of linearizing the extremal equations \eqref{extremalsyst} coming from the Pontryagin maximum principle, at the point $(\bar x,\bar\lambda,-1,\bar u)$ which is the solution of the static optimal control problem. This will be done rigorously and in details further, but let us however explain roughly this step and take the opportunity to introduce several notations useful to state our main result.

Setting $x_T(t)=\bar x+\delta x(t)$, $\lambda_T(t)=\bar\lambda+\delta\lambda(t)$ and $u_T(t)=\bar u+\delta u(t)$ (perturbation variables), we get from the third equation of \eqref{extremalsyst} that, \textit{at the first order} in the perturbation variables $(\delta x,\delta\lambda,\delta u)$, $\delta u(t) = -H_{uu}^{-1} \left( H_{xu} \delta x(t) + H_{\lambda u} \delta\lambda(t) \right)$ (in what follows we will assume that the matrix $H_{uu}$ is invertible),
and then, from the two first equations of \eqref{extremalsyst},
\begin{equation}\label{13:55}
\begin{split}
\delta \dot x(t) &= \left( H_{x\lambda} - H_{u\lambda}H_{uu}^{-1} H_{xu}\right) \delta x(t) - H_{u\lambda} H_{uu}^{-1} H_{\lambda u} \delta\lambda(t) , \\
\delta\dot\lambda(t) & = \left( -H_{xx} + H_{ux} H_{uu}^{-1} H_{xu} \right) \delta x(t) + \left( -H_{\lambda x}+H_{ux}H_{uu}^{-1}H_{\lambda u}\right) \delta\lambda(t)  .
\end{split}
\end{equation}
Here above and in the sequel, we use the following notations. The Hessian of the Hamiltonian $H$ at $(\bar x,\bar\lambda,-1,\bar u)$ is written in blocks as
$$
\mathrm{Hess}_{(\bar x,\bar\lambda,-1,\bar u)}(H) = 
\begin{pmatrix}
H_{xx} & H_{x\lambda} & H_{xu} \\
H_{\lambda x} & 0 & H_{\lambda u} \\
H_{ux} & H_{u\lambda} & H_{uu} 
\end{pmatrix},
$$
where the matrices
$$
H_{xx} = \frac{\partial^2 H}{\partial x^2}(\bar x,\bar\lambda,-1,\bar u),\qquad
H_{x\lambda} = \frac{\partial^2 H}{\partial x\partial\lambda}(\bar x,\bar\lambda,-1,\bar u),
$$
are of size $n\times n$, with $H_{x\lambda}=H_{\lambda x}^*$ (where the upper star stands for the transpose), the matrices 
$$
H_{xu} = \frac{\partial^2 H}{\partial x\partial u}(\bar x,\bar\lambda,-1,\bar u),\qquad
H_{\lambda u} = \frac{\partial^2 H}{\partial\lambda\partial u}(\bar x,\bar\lambda,-1,\bar u),
$$
are of size $n\times m$, with $H_{xu}=H_{ux}^*$ and $H_{\lambda u}=H_{u\lambda}^*$, and the matrix $H_{uu}$ is of size $m\times m$ (it will be assumed to be invertible in the main result hereafter).
Recall that we have set $\lambda^0=-1$ (multiplier associated with the cost) because we have assumed throughout that the abnormal case does not occur in our framework.

We define the matrices
\begin{equation}\label{def_A_B}
A = H_{x\lambda} - H_{u\lambda} H_{uu}^{-1} H_{xu}   , \qquad
B = H_{u\lambda}, \qquad
W = -H_{xx} + H_{ux} H_{uu}^{-1} H_{xu}.
\end{equation}
It can be noted that
$$
H_{x\lambda} = \frac{\partial^2 H}{\partial x \partial\lambda}(\bar x,\bar\lambda,-1,\bar u) = \frac{\partial f}{\partial x}(\bar x,\bar u),
$$
and that
$$
B=H_{u\lambda} = \frac{\partial^2 H}{\partial u \partial\lambda}(\bar x,\bar\lambda,-1,\bar u) = \frac{\partial f}{\partial u}(\bar x,\bar u) .
$$
Note that, setting $Z(t) = ( \delta x(t), \delta\lambda(t))^\top$, the differential system \eqref{13:55} can be written as $\dot Z(t) = M Z(t)$ (at the first order), with the matrix $M$ defined by
\begin{equation*}
M = \begin{pmatrix}
H_{x\lambda} - H_{u\lambda}H_{uu}^{-1} H_{xu}  &  - H_{u\lambda} H_{uu}^{-1} H_{\lambda u}  \\
-H_{xx} + H_{ux} H_{uu}^{-1} H_{xu}  &  -H_{\lambda x}+H_{ux}H_{uu}^{-1}H_{\lambda u}
\end{pmatrix} =
\begin{pmatrix}
A & -B H_{uu}^{-1} B^* \\
W & -A^*
\end{pmatrix}.
\end{equation*}
As explained in details further, the Hamiltonian structure of the matrix $M$ will be of essential importance in our analysis.

Note, here, that the matrices $A$ and $B$ are \textit{not} exactly the matrices of the usual linearized control system at $(\bar x,\bar u)$, which is the system $\delta\dot{x}(t) = H_{x\lambda}\, \delta x(t) + B\, \delta u(t)$. The matrix $A$ defined in \eqref{def_A_B} is rather a deformation of $H_{x\lambda}$ with terms of the second order (note that $H_{xu}=0$ in the usual LQ problem).

\medskip

Our main result is the following.

\begin{theorem}\label{thm1}
Assume that the matrix $H_{uu}$ is symmetric negative definite, that the matrix $W$ is symmetric positive definite, and that the pair $(A,B)$ satisfies the Kalman condition, that is, 
$$\mathrm{rank}(B,AB,\ldots,A^{n-1}B)=n.$$
Assume also that the point $(\bar x,\bar x)$ is not a singular point of the mapping $R$.
Finally, assume either that the norm of the Hessian of $R$ at the point $(\bar x,\bar x)$ is small enough, or that the mapping $R$ is generic.%
\footnote{Here, the genericity is understood in the following sense. Consider the set $\mathcal{X}$ of mappings $R:\R^n\times\R^n\rightarrow\R^k$, endowed with the $C^2$ topology. The generic condition is that $R\in\mathcal{X}\setminus\mathcal{S}$, where $\mathcal{S}$ is a stratified (in the sense of Whitney, see \cite{GM}) submanifold $\mathcal{S}$ of $\mathcal{X}$ of codimension greater than or equal to one.}
Then, there exist constants $\varepsilon>0$, $C_1>0$, $C_2>0$, and a time $T_0>0$ such that, if
\begin{equation}\label{defect}
\bar D = \Vert R(\bar x,\bar x)\Vert +
\left\Vert \begin{pmatrix} -\bar\lambda\\ \bar\lambda\end{pmatrix} - \sum_{i=1}^k \gamma_i \nabla R^i(\bar x,\bar x) \right\Vert \leq \varepsilon 
\end{equation}
then, for every $T>T_0$, the optimal control problem $\bf (OCP)_T$ has at least one optimal solution having a normal extremal lift $(x_T(\cdot),\lambda_T(\cdot),-1,u_T(\cdot))$ satisfying
\begin{equation}\label{estimate_turnpike}
\Vert x_T(t)-\bar x\Vert+\Vert \lambda_T(t)-\bar \lambda\Vert+\Vert u_T(t)-\bar u\Vert \leq C_1 ( e^{-C_2 t}+e^{-C_2 (T-t)}),
\end{equation}
for every $t\in[0,T]$.
\end{theorem}

\begin{remark}\label{rem1}
As follows from the proof of that result, the constant $C_2$ is defined as follows. Let $E_-$ (resp., $E_+$) is the minimal symmetric negative definite matrix (resp., maximal symmetric positive definite) solution of the algebraic Riccati equation
$$
XA + A^* X - X BH_{uu}^{-1} B^* X - W = 0.
$$
Then $C_2$ is the spectral abscissa of the Hurwitz matrix $A-BH_{uu}^{-1}B^*E_-$, that is,
$$
C_2 = - \max \{ \Re(\mu) \mid \mu\in\mathrm{Spec}(A-BH_{uu}^{-1}B^*E_-)  \}   >0 .
$$
The constant $C_1$ depends in a linear way on $\bar D$ and on $e^{-C_2 T}$. In particular, $C_1$ is smaller as $\bar D$ is smaller and/or $T$ is larger.

Note that the existence and uniqueness of $E_-$ and $E_+$ follows from the well-known algebraic Riccati theory (see, e.g., \cite{Abou-Kandil,Kwakernaak,Trelatbook}), using the assumptions that the pair $(A,B)$ satisfies the Kalman condition, that $H_{uu}$ is negative definite, and that $W$ is positive definite. Moreover, under these assumptions the matrix $A-BH_{uu}^{-1}B^*E_-$ is Hurwitz, that is, all its eigenvalues have negative real parts. 
\end{remark}

\begin{remark}
The Kalman condition, which says that the linear system $\dot{X}(t)=AX(t)+BU(t)$ is controllable, is very usual. Note however, as already said, that this linear system is not exactly the linearized system of the nonlinear control system \eqref{contsys} at the point $(\bar x,\bar u)$.
It can be noted that this Kalman controllability assumption, which is used here as one of the sufficient conditions ensuring the turnpike property, is used only to ensure the existence and uniqueness of the minimal and maximal solutions of the Riccati equation, sharing the desired spectral assumptions.
\end{remark}

\begin{remark}\label{rem2}
In the linear quadratic case (that is, with an autonomous linear system and a quadratic cost; in that case the matrices $A$, $B$ defined by \eqref{def_A_B} coincide indeed with the matrices defining the system), the result of Theorem \ref{thm1} holds true globally, that is, $\varepsilon=+\infty$. We provide all details on the LQ case in Section \ref{sec_LQ}.
\end{remark}

\begin{remark}
The assumption that the symmetric matrix $H_{uu}=\frac{\partial^2 H}{\partial u^2}(\bar x,\bar\lambda,-1,\bar u)$ be negative definite is standard in optimal control, and is usually referred to as a \textit{strong Legendre condition} (see, e.g., \cite{AgSachkov,BCT,BFT}). It implies that the implicit equation $\frac{\partial H}{\partial u}(x,\lambda,-1,u)=0$ can be solved in $u$ in a neighborhood of $(\bar x,\bar\lambda,-1,\bar u)$, by an implicit function argument.
This assumption is satisfied for instance whenever the system is control-affine and the function $f^0$ in the cost functional is a positive definite quadratic form in $(x,u)$ (see Section \ref{sec_controlaffine} for more details).
For more general nonlinear systems the strong Legendre condition is generally assumed along a given extremal in order to ensure its local (in space and time) optimality property (see, e.g., \cite{BCT}). 
\end{remark}

\begin{remark}
The assumption that the symmetric matrix $W$ be positive definite is (to our knowledge) not standard in optimal control. 
It is commented in Section \ref{sec2} through classes of examples.
In the LQ case however this assumption is natural and automatically satisfied (see Section \ref{sec_LQ}).
\end{remark}

\begin{remark}
The assumptions on the terminal conditions, represented by the mapping $R$, are generic ones. For instance these assumptions are automatically satisfied if the terminal conditions are linear, or are almost linear (which means that the norm of the Hessian of $R$ is small). As will be proved in Lemma \ref{lemQ}, if the set $R=0$ is a differential submanifold of $\R^n\times\R^n$ whose curvature at the point $(\bar x,\bar x)$ is too large, then there is a risk that in our proof some matrix be not invertible (more precisely, the matrix $Q$ defined by \eqref{defQ}), which would imply the ill-posedness of the shooting problem coming from the Pontryagin maximum principle. We prove that such a condition is however very exceptional (non-generic).
\end{remark}

\begin{remark}
The assumption \eqref{defect} that $\bar D$ be small enough means that $(\bar x,\bar\lambda)$ is \textit{almost} a solution of \eqref{terminalcond} and of \eqref{condtransv}, in the sense that
$$
R(\bar x,\bar x)\simeq 0
$$
and
$$
\begin{pmatrix} -\bar\lambda\\ \bar\lambda\end{pmatrix} \simeq \sum_{i=1}^k \gamma_i \nabla R^i(\bar x,\bar x) .
$$ 
In order  to facilitate the understanding, let us provide hereafter several typical examples of terminal conditions, following those of Remark \ref{remBC}.
\begin{itemize}
\item If the initial and final points are fixed in $\bf (OCP)_T$, then the smallness condition is satisfied as soon as the initial point $x_0$ and the final point $x_1$ are close enough to $\bar x$.

\item If the initial point is fixed in $\bf (OCP)_T$ and the final point is let free, then the smallness condition is satisfied as soon as the initial point $x_0$ is close enough to $\bar x$ and the Lagrange multiplier $\bar\lambda$ has a small enough norm. This additional requirement that $\Vert\bar\lambda\Vert$ be small enough is satisfied as soon as $(\bar x,\bar u)$ is `almost" a local or a global minimizer of the problem of minimizing $f^0$ over the whole set $\R^n\times\R^m$ (that is, without the constraint $f=0$). For instance if $f(\bar x,\bar u)=0$ and if $f^0$ is nonnegative with $f^0(\bar x,\bar u)=0$ then $\bar\lambda=0$ and then the condition is obviously satisfied.

\item If we impose that $x(0)=x(T)$ (periodicity assumption) in $\bf (OCP)_T$, then $\bar D=0$ and hence the smallness condition is always satisfied without any further requirement.
\end{itemize}
Our terminal conditions are far more general and cover a very large number of situations, whose interpretation is let to the reader.
\end{remark}

\begin{remark}
It follows from \eqref{estimate_turnpike} that, under the conditions of Theorem \ref{thm1}, we have
$$
\lim_{T\rightarrow +\infty} \frac{1}{T}\int_0^T x_T(t) \, dt = \bar x,\quad
\lim_{T\rightarrow +\infty} \frac{1}{T}\int_0^T \lambda_T(t) \, dt = \bar \lambda,\quad
\lim_{T\rightarrow +\infty} \frac{1}{T}\int_0^T u_T(t) \, dt = \bar u,
$$
and
$$
\lim_{T\rightarrow +\infty} \frac{C_T(u_T)}{T}  = f^0(\bar x,\bar u).
$$
We thus recover in particular results from \cite{PorrettaZuazua}. The latter equality says that the time-asymptotic average over the optimal values of $\bf (OCP)_T$ coincides with the optimal value of the static optimal control problem.
\end{remark}

\begin{remark}
As explained in \cite{AndersonKokotovic,WildeKokotovic}, in the case where the initial and final points are fixed in the optimal control problem, that is, $R(x,y)=(x-x_0,y-x_1)$, the optimal trajectory and control solutions of $\bf (OCP)_T$ can be approximately obtained by piecing together the solutions of two infinite-time regulator problems: the first one consists of steering asymptotically in time the initial point $x_0$ to the point $\bar x$ (stabilization problem in forward time), and the second one can be seen as a reverse in time problem, consisting of steering $x_1$ to $\bar x$ in infinite time (stabilization problem in reverse time). These initial and final phases are transient and exponentially quick, and in the long mid-interval the trajectory stays exponentially close to $\bar x$.
\end{remark}

As it is well known, the turnpike property is actually due to a general hyperbolicity phenomenon. Roughly speaking, in the neighborhood of a saddle point, any trajectory of a given hyperbolic dynamical system, which is constrained to remain in this neighborhood in large time, will spend most of the time near the saddle point. This very simple observation is at the heart of the turnpike results. Actually, when linearizing the extremal equations derived from the Pontryagin maximum principle at the steady-state $(\bar x,\bar \lambda, \bar u)$ solution of the static optimal control problem \eqref{staticpb}, we get a hyperbolic system. In other words, this steady-state, analogue of a Von Neumann point in econometry, is a saddle point for the extremal system \eqref{extremalsyst}.
In the present paper we will use as well this remark, instrumentally combined with precise estimates on Riccati equations inspired from \cite{WildeKokotovic} in order to tackle general terminal conditions.

Our analysis will consist of analyzing shrewdly the behavior of the solutions of \eqref{13:55}, written in the form of $\dot Z(t) = M Z(t)$ (at the first order).
In the proof of Theorem \ref{thm1} (which is done in Section \ref{sec3}), we will use in an instrumental way the fact that $M$ is a Hamiltonian matrix, but with a however specific feature: it is purely \textit{hyperbolic}. This will be proved thanks to fine (but classical) properties of Riccati equations.
In order to highlight the main ideas and in particular the central hyperbolicity phenomenon, we will first prove the theorem in the LQ case (see Section \ref{sec3.1}), with very simple terminal conditions. The proof of the general nonlinear case with general terminal conditions is done in Section \ref{sec3.2}, and is more technical due to two reasons: the first is that we have to be careful with the remainder terms, and the second is due to the generality of the terminal conditions under consideration. Note that it is also required, in the general case, to prove that the corresponding shooting problem is well posed, which is far from obvious (see lemmas \ref{lembeforeQ} and \ref{lemQ}).

Before coming to the proof of Theorem \ref{thm1}, in the next section we provide examples and applications of our main result.

\section{Examples and applications}\label{sec2}
In Section \ref{sec_LQ}, we focus on the particular but important case of linear quadratic problems. We explain in detail how Theorem \ref{thm1} can be stated more precisely in that case. In Section \ref{sec_controlaffine}, we focus on another important class of optimal control problems, settled with control-affine systems (linear in the control, nonlinear in the state). We also provide numerical illustrations.
In Section \ref{sec2.3}, we briefly recall what are the numerical methods that are usually implemented in order to solve numerically an optimal control problem, and recall their usual limitations in terms of initialization. In the framework of our turnpike result, we provide a new appropriate way of initializing successfully a direct or an indirect method in optimal control. In particular, we design an adequate variant of the classical shooting method.
Finally, in Section \ref{sec2.4} we provide further comments and describe some of the many open problems that arise from our study.

\subsection{The linear quadratic case}\label{sec_LQ}
In this section we assume that
$$ 
f(x,u)=Ax+Bu,
$$
with $A$ a matrix of size $n\times n$ and $B$ a matrix of size $n\times m$, and that
$$
f^0(x,u) = \frac{1}{2}(x-x^d)^*Q(x-x^d) + \frac{1}{2}(u-u^d)^*U(u-u^d),
$$
where $Q$ is a $n\times n$ symmetric positive definite matrix and $U$ is a $m\times m$ symmetric positive definite matrix, and where $x^d\in\R^n$ and $u^d\in\R^m$ are arbitrary. The matrices $Q$ and $U$ are weight matrices.
It is assumed that the pair $(A,B)$ satisfies the Kalman condition.

We consider the following terminal conditions. Let $x_0\in \R^n$ and $x_1\in\R^n$ be arbitrary.
We consider either the terminal constraints $x(0)=x_0$ and $x(T)=x_1$ (that is, initial and final points fixed), or $x(0)=x_0$ and $x(T)$ free (that is, only the initial point is fixed).

Note that, in this LQ case, one has $H_{xu}=0$ and $H_{xx}=-Q$ and hence the matrices $A$, $B$ defined by \eqref{def_A_B} coincide indeed with the above matrices defining the system. Moreover, $H_{uu}=-U$ is symmetric negative definite and $W=Q$ is symmetric positive definite by definition.

Besides, it is clear that $\bf (OCP)_T$ has a unique solution $(x_T(\cdot),u_T(\cdot))$, having a normal extremal lift $(x_T(\cdot),\lambda_T(\cdot),-1,u_T(\cdot))$ (note that the Kalman condition implies that the extremal lift is normal), and the control has the simple expression
$$ u_T(t) = U^{-1}B^*\lambda_T(t).$$
The extremal system \eqref{extremalsyst} is written as
\begin{equation}\label{extremalsystLQ}
\begin{split}
\dot{x}_T(t) &= Ax_T(t)+BB^*\lambda_T(t)+Bu^d, \qquad x_T(0)=x_0,  \\
\dot{\lambda}_T(t) &= Qx_T(t)-A^*\lambda_T(t)-Qx^d,
\end{split}
\end{equation}
for almost every $t\in[0,T]$.
In the case where the final point $x_T(T)$ is free then we have the transversality condition $\lambda_T(T)=0$.

\medskip

The static optimal control problem \eqref{staticpb} is written, in that case, as the (strictly convex) minimization problem
\begin{equation}\label{staticpbLQ}
\min_{\stackrel{(x,u)\in\R^n\times\R^m}{Ax+Bu=0}} \frac{1}{2}\left( (x-x^d)^*Q(x-x^d) + (u-u^d)^*U(u-u^d) \right).
\end{equation}
It has a unique solution $(\bar x,\bar u)$, associated with a normal Lagrange multiplier $(\bar\lambda,-1)$.
Note that the optimization problem \eqref{staticpbLQ} is indeed qualified as soon as $\mathrm{null}(A^*)\cap\mathrm{null}(B^*)=\{0\}$, condition which is implied by (and is weaker than) the Kalman condition. Therefore the abnormal case does not occur here.
The system \eqref{static_extr}, coming from the Lagrange multipliers rule, says here that there exists $\bar\lambda\in\R^n\setminus\{0\}$ such that $\bar u = u^d + U^{-1}B^*\bar\lambda$ and
\begin{equation}\label{extrLQ}
\begin{split}
A\bar x + BB^*\bar\lambda + Bu^d &= 0,\\
Q\bar x - A^*\bar\lambda - Q x^d &= 0.
\end{split}
\end{equation}

As mentioned in Remark \ref{rem2}, the result of Theorem \ref{thm1} holds true globally.
In this LQ framework, Theorem \ref{thm1} takes the following form.

\begin{theorem}\label{thm1LQ}
There exist constants $C_1>0$ and $C_2>0$ such that for every time $T>0$ the optimal control problem $\bf (OCP)_T$ has a unique solution $(x_T(\cdot),\lambda_T(\cdot),u_T(\cdot))$, which satisfies
\begin{equation}
\Vert x_T(t)-\bar x\Vert+\Vert \lambda_T(t)-\bar \lambda\Vert+\Vert u_T(t)-\bar u\Vert \leq C_1 ( e^{-C_2 t}+e^{-C_2 (T-t)}),
\end{equation}
for every $t\in[0,T]$.
\end{theorem}

\begin{remark}\label{rem1_1}
To be more precise with the constants, what we establish is that
\begin{equation*}
\begin{split}
\Vert x_T(t)-\bar x\Vert &\leq \Vert x_0-\bar x \Vert e^{-C_2 t} + \Vert E_-^{-1}\bar\lambda\Vert e^{-C_2 (T-t)}   \\
&\qquad + \mathrm{O}\left( \Vert E_-^{-1}\bar\lambda \Vert e^{-C_2 (t+T)} + \Vert E_-^{-1}E_+\Vert \Vert x_0-\bar x\Vert e^{-C_2 (2T-t)} \right) , \\
\Vert \lambda_T(t)-\bar\lambda\Vert &\leq \Vert E_+\Vert \Vert x_0-\bar x \Vert e^{-C_2 t} +\Vert E_-\Vert  \Vert E_-^{-1}\bar\lambda\Vert e^{-C_2 (T-t)} \\
&\qquad + \mathrm{O}\left( \Vert E_+\Vert \Vert E_-^{-1}\bar\lambda \Vert e^{-C_2 (t+T)} + \Vert E_-\Vert  \Vert E_-^{-1}E_+\Vert \Vert x_0-\bar x\Vert e^{-C_2 (2T-t)}  \right)  , \\
\Vert u_T(t)-\bar u\Vert &\leq \Vert U^{-1}\Vert\Vert B\Vert  \Vert \lambda_T(t)-\bar\lambda\Vert ,
\end{split}
\end{equation*}
for every $t\in[0,T]$, where $E_-$ (resp., $E_+$) is the minimal symmetric negative definite matrix (resp., maximal symmetric positive definite) solution of the algebraic Riccati equation\footnote{Note that their existence and uniqueness follows from the well-known algebraic Riccati theory (see, e.g., \cite{Abou-Kandil,Kwakernaak,Trelatbook}), since the pair $(A,B)$ satisfies the Kalman condition, and $U$ and $Q$ are positive definite. Moreover the matrix $A+BU^{-1}B^*E_-$ is Hurwitz, that is, all its eigenvalues have negative real parts.}
$$
XA + A^* X + X BU^{-1} B^* X - Q = 0,
$$
and where $C_2$ is the spectral abscissa of the Hurwitz matrix $A+BU^{-1}B^*E_-$, that is,
$$
C_2 = - \max \{ \Re(\mu) \mid \mu\in\mathrm{Spec}(A+BU^{-1}B^*E_-)  \}   >0 .
$$
Here, the remainder terms $\mathrm{O}(\cdot)$ are to be understood with respect to $T$ large.
\end{remark}

\begin{remark}
Let us comment on the pair of points $(x^d,u^d)$, which have been arbitrarily fixed at the beginning. 

First of all, let us consider the particular case where $(x^d,u^d)$ is an equilibrium point, that is, $Ax^d+Bu^d=0$.
In that case, $(\bar x,\bar u)=(x^d,u^d)$ is the solution of the static optimal control problem \eqref{staticpb}, and $\bf (OCP)_T$ is a usual linear-quadratic problem. It is very well known that when $T=+\infty$ then the (unique) solution of $\bf (OCP)_{\infty}$ is given by the algebraic Riccati theory: the optimal control is $u_{\infty}(t) = u^d + U^{-1}B^*E_- x_{\infty}(t)$, where $E_-$ is defined as in Remark \ref{rem1_1}, and $x_{\infty}(t)$ converges exponentially to $x^d$ as $t$ tends to $+\infty$.

The turnpike property says here that the optimal trajectory is approximately made of three pieces, the first of which consists of passing exponentially quickly from $x_0$ to $x^d$, then of staying most of the time at the steady-state $x^d$, and the last piece consists of passing exponentially quickly from $x_d$ to $x_1$.
We thus recover exactly the result of \cite{WildeKokotovic,AndersonKokotovic}.

Note that, in \cite{PorrettaZuazua}, the final point is let free. In that case the transversality condition at the final time gives $\lambda_T(T)=0$, and in the turnpike structure described above there is no third piece anymore as soon as $(x^d,u^d)$ is an equilibrium point.

\medskip

Secondly, let us now assume that $(x^d,u^d)$ is not an equilibrium point. Then we are not anymore within the framework of \cite{WildeKokotovic,AndersonKokotovic}. When $T$ tends to $+\infty$, the optimal solution $(x_T(\cdot),u_T(\cdot))$ does not converge towards $(x^d,u^d)$ (which is not an equilibrium). What the result says is that the optimal extremal $(x_T(\cdot),\lambda_T(\cdot),u_T(\cdot))$ spends most of its time close to $(\bar x,\bar\lambda,\bar u)$, where $(\bar x,\bar u)$ is the nearest (for the norms induced by $Q$ and $U$) equilibrium point to $(x^d,u^d)$. We recover here the result of \cite{PorrettaZuazua}.

Note that $C_T(u_T)$ tends to $+\infty$ as $T$ tends to $+\infty$, as soon as $(x^d,u^d)$ is not an equilibrium point. Actually, one has
$$
\lim_{T\rightarrow +\infty} \frac{C_T(u_T)}{T} = \frac{1}{2} \left( (\bar x-x^d)^*Q(\bar x-x^d) + (\bar u-u^d)^*U(\bar u-u^d) \right).
$$
\end{remark}

\begin{example}\label{example1}
Let us provide a simple example in order to illustrate the turnpike phenomenon in the LQ case. Consider the two-dimensional control system
\begin{equation*}
\begin{split}
\dot{x}_1(t) &= x_2(t), \\
\dot{x}_2(t) &= -x_1(t)+ u(t) ,
\end{split}
\end{equation*}
with fixed initial point $(x_1(0),x_2(0))=(0,0)$, and the problem of minimizing the cost functional
$$
\frac{1}{2} \int_0^T \left( \left(x_1(t)-2\right)^2+\left(x_2(t)-7\right)^2+u(t)^2 \right) dt .
$$
The final point is let free.
An easy computation shows that the optimal solution of the static problem is given by $\bar x = (1,0)$, $\bar u=1$, and $\bar \lambda=(-7,1)$.

We compute the optimal solution $(x_1(\cdot),x_2(\cdot),\lambda_1(\cdot),\lambda_2(\cdot),u(\cdot))$ in time $T=30$, by using a direct method of optimal control (see \cite{Trelatbook,TrelatJOTA}). More precisely we discretize the above optimal control problem by using a simple explicit Euler method with $1000$ time steps, and we use the optimization routine \texttt{IPOPT} (see \cite{IPOPT}) combined with the automatic differentiation code \texttt{AMPL} (see \cite{AMPL}) on a standard desktop machine. The result is drawn on Figure \ref{figLQ}.
\begin{figure}[h]
\centerline{\includegraphics[width=14cm]{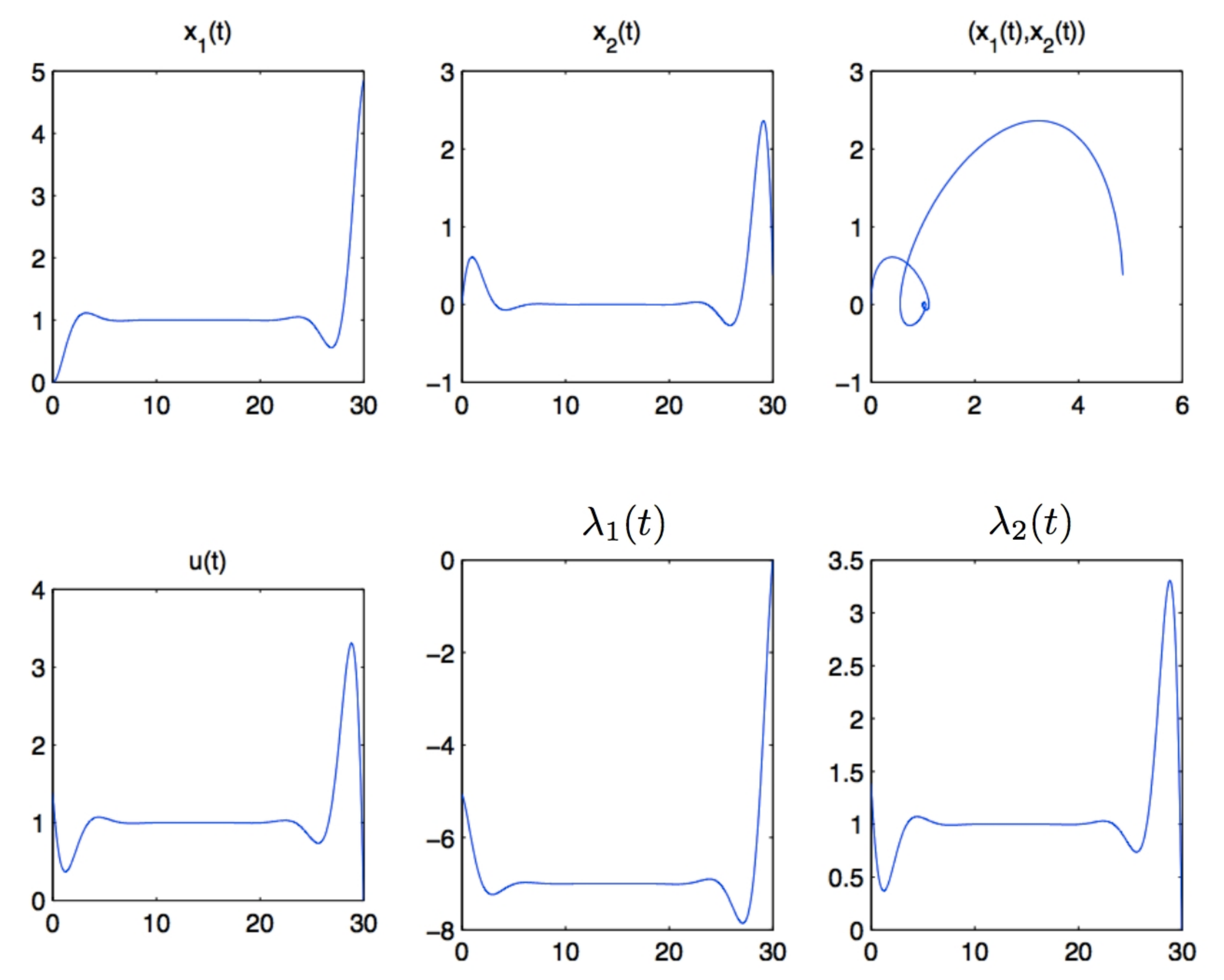}}
\caption{Example in the LQ case.}\label{figLQ}
\end{figure}
Note that, since the final point is free, the transversality condition yields $\bar \lambda_T(T)=0$. Besides, the maximization condition of the Pontryagin maximum principle implies that $u(t)=\lambda_2(t)$.

The turnpike property can be observed on Figure \ref{figLQ}. As expected, except transient initial and final arcs, the extremal $(x_1(\cdot),x_2(\cdot),\lambda_1(\cdot),\lambda_2(\cdot),u(\cdot))$ remains close to the steady-state $(1,0,1,-7,1)$.

It can be noted that, along the interval of time $[0,30]$, the curves $x_1(\cdot)$, $x_2(\cdot)$, $\lambda_1(\cdot)$, $\lambda_2(\cdot)$ and $u(\cdot)$ oscillate around their steady-state value (with an exponential damping). This oscillation is visible on Figure \ref{figLQoscill}, where one can see the successive (exponentially small) loops that $(x_1(\cdot),x_2(\cdot))$ makes around the point $(1,0)$. The number of loops tends to $+\infty$ as the final time $T$ tends to $+\infty$.
\begin{figure}[H]
\centerline{\includegraphics[width=8cm]{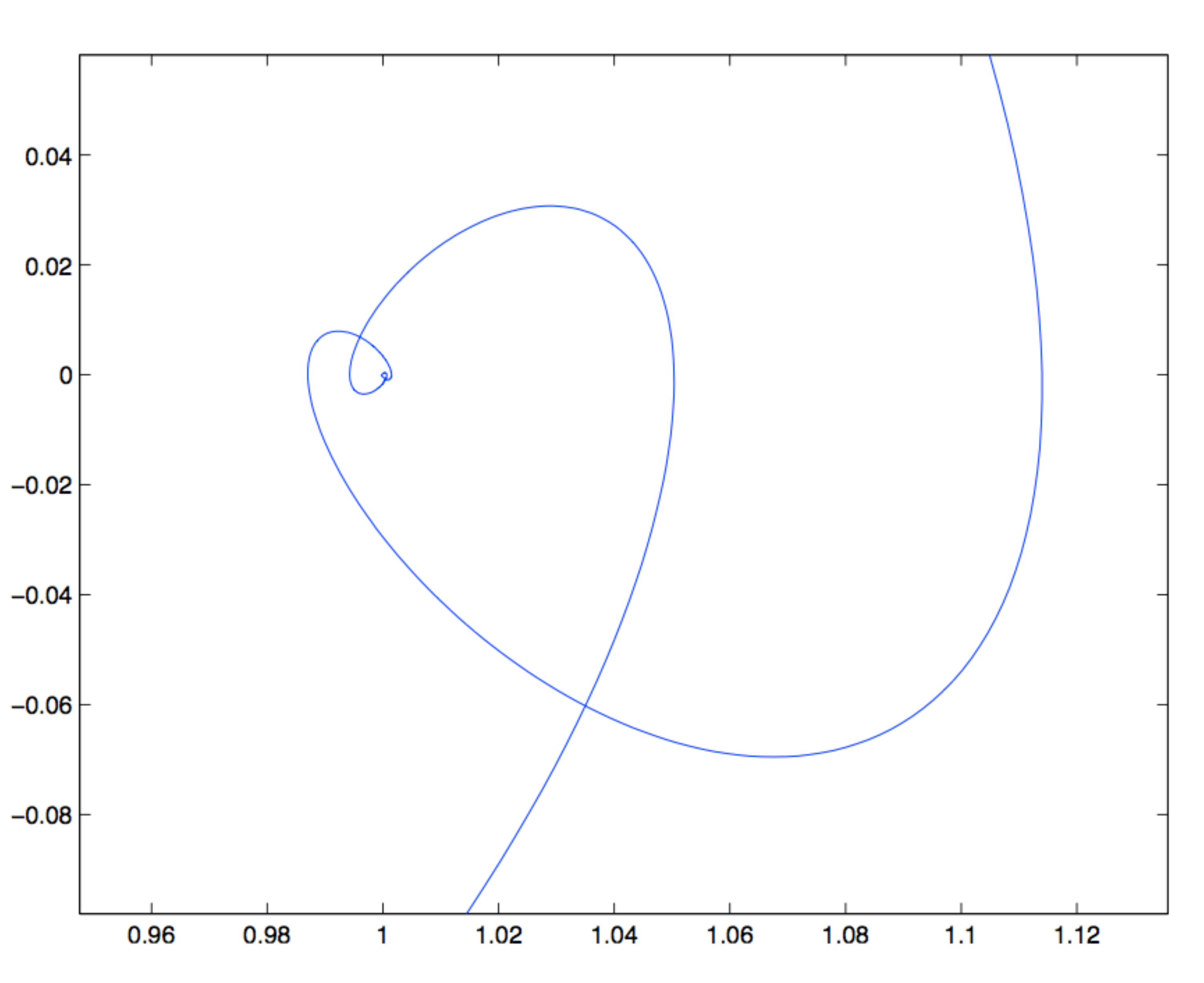}}
\caption{Oscillation of $(x_1(\cdot),x_2(\cdot))$ around the steady-state $(1,0)$.}\label{figLQoscill}
\end{figure}
\end{example}

\subsection{Control-affine systems with quadratic cost}\label{sec_controlaffine}
In this section, we consider the class of control-affine systems with quadratic cost, that is,
$$
f(x,u) = f_0(x) + \sum_{i=1}^m u_i f_i(x),
$$
where $f_i$ is a $C^2$ vector field in $\R^n$, for every $i\in\{0,\ldots,m\}$, and 
$$
f^0(x,u) = \frac{1}{2} (x-x^d)^*Q(x-x^d) + \frac{1}{2} (u-u^d)^*U(u-u^d),
$$
with $Q$ a $n\times n$ symmetric positive definite matrix, and $U$ a $m\times m$ symmetric positive definite matrix. The matrices $Q$ and $U$ are weight matrices, as in the LQ case.
In this framework, we have
$$
H = \langle\lambda,f_0(x)\rangle + \sum_{i=1}^m u_i\langle\lambda,f_i(x)\rangle - \frac{1}{2} (x-x^d)^*Q(x-x^d) - \frac{1}{2} (u-u^d)^*U(u-u^d),
$$
$$
H_{xu} = \begin{pmatrix}
\langle\bar\lambda,df_1(\bar x)\rangle, \ldots, \langle\bar\lambda,df_m(\bar x)\rangle,
\end{pmatrix}
$$
$$
H_{xx} = - Q + \langle\bar\lambda,d^2f_0(\bar x)\rangle + \sum_{i=1}^m u_i \langle\bar\lambda,d^2f_i(\bar x)\rangle,
$$
$$
H_{uu} = -U,
$$
and hence
$$
W = -H_{xx}+H_{ux}H_{uu}^{-1}H_{xu}
= Q - H_{xu}^*U^{-1}H_{xu} - \langle\bar\lambda,d^2f_0(\bar x)\rangle - \sum_{i=1}^m u_i \langle\bar\lambda,d^2f_i(\bar x)\rangle .
$$
Intuitively, the requirement that $W>0$ says that the positive weight represented by $Q$ has to be large enough in order to compensate possible distortion by the vector fields. Note that if the controlled vector fields $f_i$, $i=1,\ldots,m$, are linear in $x$, then $W=Q$ and hence the assumption on $W$ is automatically satisfied. This is the case whenever the dynamics have the form $f(x,u)=A_0 x+\sum_{i=1}^m u_i A_i x$.

\begin{example}\label{example2}
Let us provide a simple example in order to illustrate the turnpike phenomenon for a control-affine system with a quadratic cost. Consider the optimal control problem of steering the two-dimensional control system
\begin{equation*}
\begin{split}
\dot{x}_1(t) &= x_2(t), \\
\dot{x}_2(t) &= 1-x_1(t)+x_2(t)^3 + u(t) ,
\end{split}
\end{equation*}
from the initial point $(x_1(0),x_2(0))=(1,1)$ to the final point $(x_1(T),x_2(T))=(3,0)$, by minimizing the cost functional
$$
\frac{1}{2} \int_0^T \left( \left(x_1(t)-\frac{1}{2}\right)^2+\left(x_2(t)-\frac{1}{2}\right)^2+\left(u(t)-1\right)^2 \right) dt .
$$
This is a nonlinear harmonic oscillator with an explosive cubic term.
An easy computation shows that the optimal solution of the static problem is given by $\bar x = (\frac{5}{4},0)$, $\bar u=\frac{1}{4}$, and $\bar \lambda=(-\frac{1}{2},-\frac{3}{4})$.

As in the example \ref{example1}, we compute the optimal solution $(x_1(\cdot),x_2(\cdot),\lambda_1(\cdot),\lambda_2(\cdot),u(\cdot))$ in time $T=20$, by using a direct method.
The result is drawn on Figure \ref{figaffine}.
\begin{figure}[h]
\centerline{\includegraphics[width=14cm]{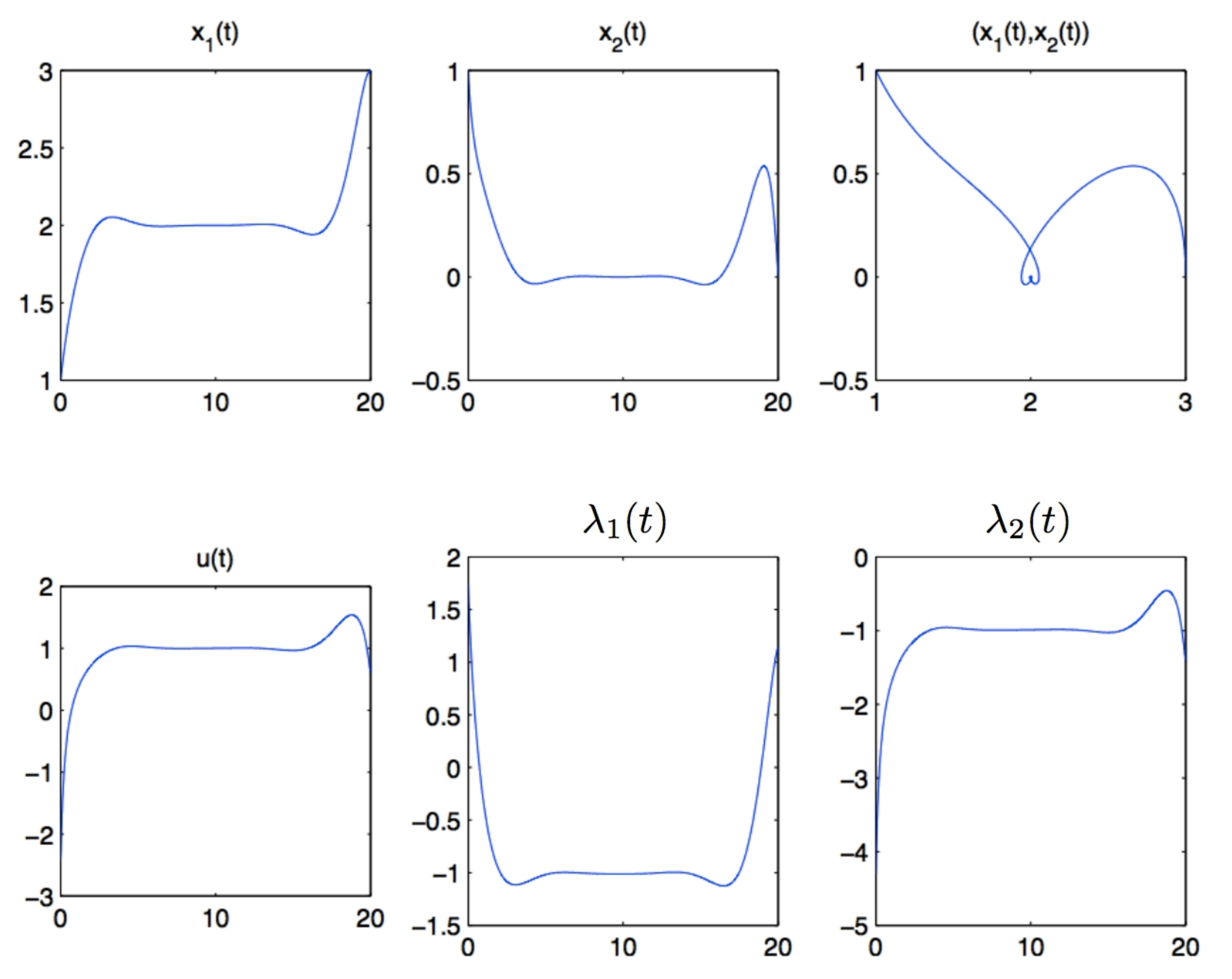}}
\caption{Example in the control-affine case.}\label{figaffine}
\end{figure}
Note that, according to the maximization condition of the Pontryagin maximum principle, we have $u(t)=2+\lambda_2(t)$.
The turnpike property can be observed on Figure \ref{figaffine}. As expected, except transient initial and final arcs, the extremal $(x_1(\cdot),x_2(\cdot),\lambda_1(\cdot),\lambda_2(\cdot),u(\cdot))$ remains close to the steady-state $(2,0,-1,-1,1)$.

Note that we have the same (exponentially damped) oscillation phenomenon as in the example \ref{example1} around the steady-state.
This oscillation can be seen on Figure \ref{figaffineoscill}, in the form of successive (exponentially small) heart-shaped loops that $(x_1(\cdot),x_2(\cdot))$ makes around the point $(2,0)$. The number of loops tends to $+\infty$ as the final time $T$ tends to $+\infty$.

It can be noted that, due to the explosive term $x_2^3$, the convergence of the above optimization problem may be difficult to ensure. However, as we will explain in Section \ref{sec2.3}, we use here the particularly adequate initialization given by the solution of the static problem. Then the convergence is easily obtained. The convergence of an optimization solver with any other initialization would certainly not be ensured.

\begin{figure}[H]
\centerline{\includegraphics[width=8cm]{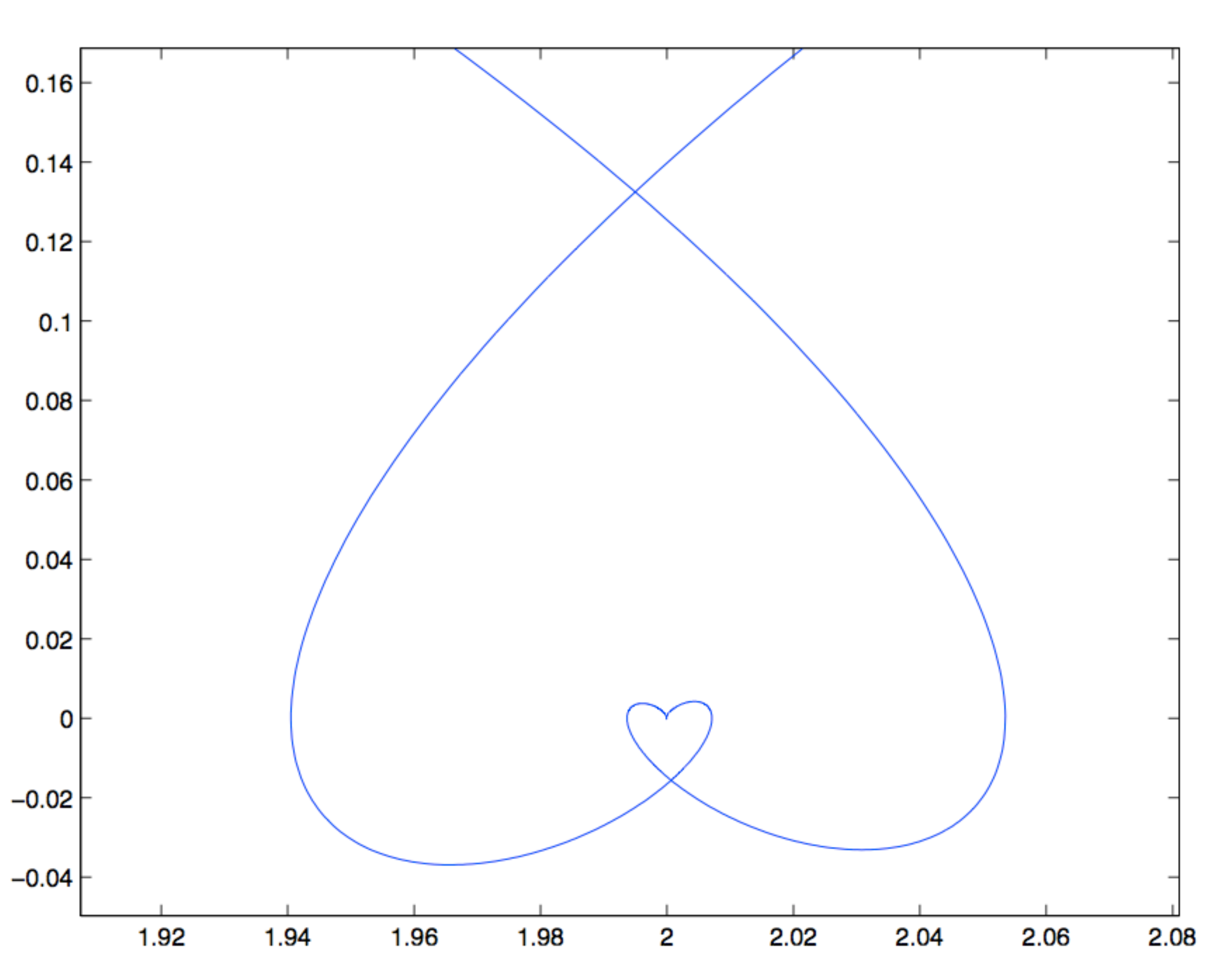}}
\caption{Oscillation of $(x_1(\cdot),x_2(\cdot))$ around the steady-state $(2,0)$.}\label{figaffineoscill}
\end{figure}
\end{example}

\subsection{Turnpike and numerical methods in optimal control}\label{sec2.3}
Let us first recall that there are mainly two kinds of numerical approaches in optimal control: direct and indirect methods. Roughly speaking, direct methods consist of discretizing the state and the control so as to reduce the problem to a constrained nonlinear optimization problem. Indirect methods consist of solving numerically the boundary value problem derived from the application of the Pontryagin maximum principle (shooting method).

Both methods suffer from a difficulty of initialization, the question being: how to initialize adequately the unknowns of the problem, in order to make converge successfully the numerical method?

Here, in the context of our turnpike theorem, we provide a new and natural way to ensure a successful initialization, for both direct and indirect approaches.

\subsubsection{Direct methods}
Direct methods consist of discretizing both the state and the control. After discretizing, the optimal control problem is reduced to a nonlinear optimization problem in finite dimension, or nonlinear programming problem, of the form
\begin{equation}\label{pboptim}
 \min_{Z\in C} F(Z), 
\end{equation}
where $Z=(x_1,\ldots,x_N,u_1,\ldots,u_n)$, and
\begin{equation}\label{defC}
C = \{ Z\ \vert\  g_i(Z)=0,\ i\in{1,\ldots,r},\ g_j(Z)\leq 0,\ j\in{r+1,\ldots,m} \} .
\end{equation}
There exists an infinite number of variants, depending on the choice of finite-dimensional representations of the control and of the state, of the discretization of the extremal differential equations, and of the discretization of the cost functional. We refer to \cite{Betts} for a thorough description of many direct approaches in optimal control.

Then, to solve the optimization problem \eqref{pboptim} under the constraints \eqref{defC}, there is also a large number of possible methods. We refer the reader to any good textbook of numerical optimization.

It can be noted that, in the previous examples \ref{example1} and \ref{example2}, we have used such a direct approach, and used the sophisticated interior-point optimization routine \texttt{IPOPT} combined with automatic differentiation (modeling language \texttt{AMPL}).

In any case, whatever method one can use, the immediate difficulty one is faced with is the problem of initializing the unknowns of the problem.
We propose here the following very natural idea. Assume that we are dealing with an optimal control problem like $\bf (OCP)_T$, where the final time is quite large. Assume that we are in the conditions of Theorem \ref{thm1}.
Then the optimal trajectory enjoys the turnpike property, and as proved in Theorem \ref{thm1}, the whole extremal is close to a certain stationary value which can be computed by solving the static optimal control problem \eqref{staticpb}.
This information can actually be used in an instrumental way in order to initialize successfully a numerical direct method to solve $\bf (OCP)_T$, by providing a high-quality initial guess which is then expected to make  the numerical method converge, at least if the final time $T$ is large enough.

This is exactly what we have observed in the examples \ref{example1} and \ref{example2}, where the direct method that we have implemented was converging very easily and efficiently with that appropriate initialization. In the example \ref{example2}, due to the explosive term $x_2^3$, the interest of this adequate initialization is particularly evident.

\subsubsection{Shooting method}
Let us first recall the principle of the usual shooting method.
Assume that the Pontryagin maximum principle has been applied, that there is no abnormal extremal, and that the extremal controls have been expressed, using the maximization condition, in function of $(x,p)$. Then the extremal system \eqref{extremalsyst} is reduced to a differential system of the form
\begin{equation}\label{sysF}
\dot{z}(t) = F(z(t)),
\end{equation}
with $z=(x,p)$, and the terminal conditions \eqref{terminalcond}, combined with the corresponding transversality conditions \eqref{condtransv}, can be written as
\begin{equation}\label{BCF}
G(z(0),z(T))=0.
\end{equation}
In the usual implementation of the shooting method (see, e.g., \cite{StoerBulirsch,Trelatbook,TrelatJOTA}), the unknown $z(0)$ is searched such that the solution of \eqref{sysF}, starting at $z(0)$ at time $t=0$, satisfies \eqref{BCF}.

It can be noted that we have only $n$ unknowns. Indeed in $z(0)\in\R^{2n}$, a part of dimension $n$ is already fixed. To be clear, the most usual case is when the initial and final states are fixed in the optimal control problem under consideration. In that case, $x(0)$ is already known and in $z(0)$ the unknowns are the $n$ last coordinates, that are the initial adjoint vector $p(0)$. In the shooting method, these $n$ unknowns must be tuned so that the relation $x(T)=x_1$ holds true.

The implementation is usually done using a Newton method, or some variant of it. The shooting method is then nothing else but the combination of a Newton method with a numerical method for integrating an ordinary differential equation.

As it is well known, the shooting method is in general very hard to initialize, due to the fact that the domain of convergence of the Newton method underneath is small. In order to guarantee the convergence of the shooting method, one is then required to provide an adequate initialization of $z(0)$, precise enough so that the Newton method will converge. This task may be very hard unless one does not have a rough idea of the value of $z(0)$.
Shooting methods are in general more sensitive to the initialization than direct methods.
Many remedies do exist however, that can be used for classes of problems in such or such situation (see, e.g., the survey \cite{TrelatJOTA}).

We propose here the following remedy. Assume, as before, that we are in the context of our turnpike result (Theorem \ref{thm1}).
Then not only the trajectory and the control but also the adjoint vector are close to the steady-state solution of the static optimal control problem \eqref{staticpb}.

This closedness cannot a priori be used directly to ensure the convergence of the shooting method described above, if it is implemented in the usual way. Indeed, in the context of our turnpike result, the extremal is approximately known along the interval $[\varepsilon,T-\varepsilon]$, for some $\varepsilon>0$, but it is not known at the terminal points $t=0$ and $t=T$.

The natural idea is then to modify the usual implementation of the shooting method, and to initialize it at some arbitrary point of $[\varepsilon,T-\varepsilon]$, for instance, at $t=T/2$.
The method is then the following.

\paragraph{Variant of the shooting method.}
The unknown is $z(T/2)\in\R^{2n}$. It will be naturally initialized at $(\bar x,\bar p)$, the steady-state solution of the static optimal control problem \eqref{staticpb}. Then:
\begin{itemize}
\item we integrate backwards the system \eqref{sysF}, over $[0,T/2]$, and we get a value of $z(0)$;
\item we integrate forward the system \eqref{sysF}, over $[T/2,T]$, and we get a value of $z(T)$.
\end{itemize}
Then the unknown of $z(T/2)$ must be tuned (through a Newton method) so that \eqref{BCF} is satisfied.

This very simple variant of the usual shooting method appears to be very efficient, at least when one is in the context of a turnpike.

For the optimal control problem studied in the example \ref{example2}, it is interesting to observe that this approach works perfectly and is very much stable, whereas it is extremely difficult to ensure the convergence of the usual shooting method, already for $T=2$. Actually, using very refined continuation processes as in \cite{TrelatJOTA}, we were able to make it converge for $T=10$, but the method becomes so much sensitive that it is impossible to go beyond (once again, due to the explosive term $x_2^3$ it becomes impossible to find a good initial guess in the classical shooting method whenever $T$ becomes too large).

\begin{remark}
It can be noted that this variant of the shooting method is similar to some methods used for computing traveling waves solutions of constant speed of nonlinear reaction-diffusion equations, or more generally heteroclinic orbits of infinite-dimensional dynamical systems (see, e.g., \cite{DoedelFriedman_JCAM1989,LentiniKeller_SINUM1980}). There, the turnpike is understood by the passage (phase transition) close to an equilibrium point from the stable to the unstable manifold.
\end{remark}

\subsection{Further comments and open problems}\label{sec2.4}
The turnpike property established in this article ensures that, for  general finite-dimensional optimal control problems settled in large time, the optimal control and trajectories are, most of the time, exponentially close to the optimal control and state of the corresponding steady-state (or static) optimal control problem, provided the time-horizon is large enough.

It can be noticed that, in the present article we have investigated the behavior of the solutions only near one steady-state. What can happen globally whenever there are several steady-state solutions of the static problem is related with the global dynamics and can be challenging to analyze. It is very interesting to mention the works \cite{Rapaport,Rapaport2} in which the authors characterize the optimality of several turnpikes that are in competition, for a specific class of optimal control problems. This requires a fine knowledge of the global properties of the dynamics underneath, in particular the homoclinic and heteroclinic connections and how steady-state controls can act on them.

In practice the turnpike property allows performing a significant simplification on the analysis and  computation of time-dependent optimal controls and trajectories. Namely, in view of this result, one can simply consider the steady-state problem, dropping the time dependence, and take the corresponding steady-state optimal control and state as an approximation of the time evolution ones. According to the results of this paper, we know that such an approximation is legitimate, during most of the time-horizon, except for two exponential boundary layers at the initial and final times, provided the time-horizon for the control problem is large enough. Of course, in practice, it is a very interesting issue to develop methods and principles allowing one to determine whether or not, for a time-dependent optimal control problem, the time-horizon is large enough so that the turnpike property applies. The methods developed in this paper provide estimates that can be made explicit on specific examples, yielding some safety bounds.

This principle of replacing the time-dependent optimal control problem by the steady-state one is often used in practical applications without actually proving rigorously the turnpike property. This is for instance typically the case in optimal shape design problems in \textit{Continuum Mechanics}. Indeed, both in elasticity (see \cite{Allaire}) and aeronautics (see \cite{jameson,pironneau}), most often, optimal shape designs or optimal materials are determined on the basis of a steady-state modeling. Justifying this reduction in the context of nonlinear PDE's is a very difficult and mainly open subject.  Practitioners often focus on the development of efficient numerical algorithms, combining continuous and discrete optimization techniques, Hadamard shape derivatives, topological derivatives and level set methods, homogenization theory, etc. But very little is known about the rigorous actual proximity of the time-dependent optimal shapes or materials and the steady-state ones.

Let us however comment on some of the existing literature in this important subject.

In \cite{Balas}, the author studies the problem of adjusting the steady-state shape of a large antenna near a desired profile, by means of optimal control. The antenna is modeled as a second-order in time distributed parameter system. The author shows the convergence of quasi-static optimal controllers designed from a finite-dimensional approximation towards optimal controllers of the infinite-dimensional optimal quasi-static control problem. There is however no investigation of how close the time-dependent optimal shapes (which are expected to evolve slowly in time, in a quasi-static way) are from the designed steady-state shapes.

Recently, in the context of the identification of optimal materials for heat processes, in \cite{AM} it was proved that, for large-time optimization horizons, such processes can be approximated by the optimal steady-state ones. Note however that, in the analysis in \cite{AM}, the materials (modeled by the coefficients of the second-order operator generating the parabolic dynamics) were chosen to be time-independent. Thus, this is not, strictly speaking, a turnpike result but rather a  $\Gamma$-convergence one, ensuring the convergence of optimizers from parabolic towards elliptic. 

Similar results were proved in \cite{PZ2} in the context of the control of the semilinear heat equation. In that paper it is shown that, while proving the $\Gamma$-convergence of time-independent controls of the heat equation towards the elliptic one can be carried out in a standard manner, as a consequence of the exponential convergence of parabolic trajectories towards elliptic solutions as time tends to infinity, the turnpike property is much harder to achieve. In fact the results in \cite{PZ2} about turnpike require smallness conditions on the steady-states and controls under consideration that could well be of a purely technical nature.

Note that the results of the present paper, established in a finite-dimensional setting, are based on a careful and subtle analysis of the hyperbolicity structure of the Hamiltonian system associated to the optimality system characterizing the optimal states and controls for the time-evolution problem. The extension of this analysis to the infinite-dimensional setting is a challenging open problem as it is probably a necessary step for a better understanding of the turnpike property for nonlinear PDE's and to avoid the possibly technical smallness assumptions in  \cite{PZ2}.

The corresponding linear PDE theory was developed in \cite{PorrettaZuazua}. There it was emphasized how and why the turnpike property requires the controllability of the system to be fulfilled, something which is often ignored in applications, where the turnpike property is assumed to hold as a simple consequence of the stability of the forced dynamics towards the steady-state one in large time. It would be interesting to analyze, from a qualitative point of view, to which extent such a principle holds in practice, i.e., to which extent the control problems inherit the turnpike property out of more classical stability properties of the dynamical systems in large time. The analysis in \cite{PorrettaZuazua} and also in the present paper use in a key manner the controllability properties of the underlying dynamics. 

The idea of approximating large-time dependent control problems by steady-state formulations has also been used in order to derive controllability results for difficult unstable PDE control problems (see \cite{CoronTrelat_SICON2004} for semilinear explosive heat equations, \cite{CoronTrelat_CCM2006} for semilinear explosive wave equations, \cite{SchmidtTrelat} for Couette flows with Navier-Stokes equations). This idea is also related  to \textit{adiabatic transformations} or to\textit{quasi-static deformations} (note that adiabatic controls were implemented in \cite{Boscain} for a quantum control problem).

The notion of adiabatic process comes from thermodynamics, where the models used are stationary because the phase transitions can be considered as instantaneous. Similar considerations are done in many other domains. For instance in ferromagnetic materials the phase transitions of the magnetization vector are very quick, so that a good knowledge on the system can be acquired from a static description (called micromagnetics) of the materials (see, e.g., \cite{LabbePrivatTrelat}). This is also often the case in fluid mechanics where, at least in the absence of (unsteady) turbulence, the models considered are often steady or laminar flows.

In the present paper we have also presented a number of numerical simulations that exhibit how  the turnpike property clearly emerges. This raises the interesting issue of the actual convergence of the numerical approximations performed, both by direct and by shooting methods. A closely related issue would be that of the turnpike property for the discrete versions of the continuous dynamical systems under consideration and also the possible convergence and proximity of the turnpike trajectories and controls as the time-step of the discretization tends to zero. The turnpike property has been investigated for discrete finite-dimensional dynamical systems (see \cite{CarlsonHaurieLeizarowitz_book1991,Grunediscrete}) but, as far as we know, the limit process as the mesh-size tends to zero has not been analyzed in its whole generality. When doing this, necessarily, several parameters, $T$ and the mesh-size, in particular, may interact in various manners depending on how fast $T$ tends to infinity, while the mesh-size parameter tends to zero and vice-versa. One could expect the hyperbolic structure of the linearized optimality system exhibited in this paper to be quite robust. This could allow  transferring the turnpike property from the continuous ODEs to numerical schemes, in a general framework. Note however that, in view of the fact that we are dealing with long time intervals, very likely, the numerical schemes employed will need to fulfill the property of absolute stability so that the asymptotic qualitative properties of the ODE are preserved. Finally, let us recall that, at the PDE level, the numerical approximation of control problems  is well known to be a very sensitive issue, in particular for systems governed by hyperbolic PDEs, in which spurious numerical high frequencies oscillations may destroy the controllability properties of the continuous model (see \cite{ZSIREV}).

\medskip

Let us conclude by formulating more precisely the turnpike problem in the context of  finite-dimensional optimal design. Consider the system
\begin{equation*}
\dot{x}(t) = A(t) x(t)+b.
\end{equation*}
The equation under consideration is affine, the applied force $b$ being given and time-independent. 
The control problem itself is of bilinear nature since the control is assumed to take place in the time-dependent coefficients of $A(t)$. To fix ideas, we can assume that the matrices $A(t)$, for  $0 \leq t \leq T$, belong to a class ${\cal C}$ of symmetric definite positive matrices, with eigenvalues between two lower and upper bounds, $0 < \alpha_- < \alpha_+ < \infty$. We may then consider a simple minimization criterion
\begin{equation*}
C_T(u) = \int_0^T \left( \Vert x(t) -x^d\Vert^2 + \Vert A(t)\Vert^2 \right) dt
\end{equation*}
where the target $x^d$ is given as well.

A similar problem can be formulated in the steady-state regime where the state equation is simply 
\begin{equation*}
A x+b=0
\end{equation*}
and the functional to be minimized is
\begin{equation*}
\Vert x -x^d\Vert^2 + \Vert A\Vert^2,
\end{equation*}
within the same class ${\cal C}$ of matrices $A$.

The question then concerns whether the optimal time-dependent coefficients of $A_T(t)$,  the optimal matrix $A_T$ in the time interval $[0, T]$, approximate the those of the optimal steady state one $A^*$,  as the time-horizon $T$ is large enough.

Similar questions can be formulated in the PDE setting. We emphasize that the analog of the case considered in \cite{AM} in the present finite-dimensional setting, would correspond to the situation where  the admissible matrices are time-independent. The problem is open in that parabolic setting when coefficients are allowed to depend both in space and time.

Note also that classical problems of optimal shape design for PDE's can be formulated in a similar setting since, most often, using shape deformations, the analysis is limited to considering classes of admissible elliptic operators on a given reference shape. Of course, also at the level of shape optimization, a huge difference arises depending on whether one considers time-dependent or time-independent shapes.

\medskip

\section{Proof of Theorem \ref{thm1}}\label{sec3}
\subsection{Proof in the linear quadratic case}\label{sec3.1}
Since the proof in the general case is quite technical, in order to facilitate the understanding and highlight the idea of the hyperbolicity phenomenon, we first prove the theorem in the linear quadratic case, that is, we prove Theorem \ref{thm1LQ}.
Although the framework is more particular than in Theorem \ref{thm1}, this proof has the advantage of highlighting the main idea underlying the turnpike property, which relies on a simple hyperbolicity property.

\medskip

First of all, note that the equations \eqref{extrLQ} yield the linear system
\begin{equation}\label{eqxlambda}
\begin{pmatrix}
A & BU^{-1}B^* \\ Q & -A^*
\end{pmatrix}
\begin{pmatrix}
\bar x \\ \bar\lambda
\end{pmatrix}
=
\begin{pmatrix}
-Bu^d \\ Qx^d
\end{pmatrix} .
\end{equation}
In what follows we set
\begin{equation}\label{defMLQ}
M = \begin{pmatrix}
A & BU^{-1}B^* \\ Q & -A^*
\end{pmatrix} .
\end{equation}

\begin{lemma}\label{lemuniq}
Assume that
\begin{equation}\label{hypnull}
\mathrm{null}(A^*)\cap\mathrm{null}(B^*)=\{0\}.
\end{equation}
Then the matrix $M$ is invertible and therefore the equation \eqref{eqxlambda} has a unique solution.
\end{lemma}

\begin{proof}
Take $\begin{pmatrix}x\\ y\end{pmatrix}$ in the nullspace of $M$. Then $Ax+BU^{-1}B^*y=0$ and $Qx-A^*y=0$, whence $(AQ^{-1}A^*+BU^{-1}B^*)y=0$, and therefore $\Vert Q^{-1/2}A^*y\Vert^2+\Vert U^{-1/2}B^*y\Vert^2=0$. The conclusion follows.
\end{proof}

\begin{remark}
If the pair $(A,B)$ satisfies the Kalman condition (which is well known to be a necessary and sufficient condition for the controllability of the linear system $\dot{x}=Ax+Bu$)
then the assumption \eqref{hypnull} is satisfied. The assumption \eqref{hypnull} is weaker than the Kalman condition.
\end{remark}

\begin{remark}
Actually it is easy to see that $\mathrm{rank}(M) = n + \mathrm{rank}(AQ^{-1/2}A^*+BU^{-1/2}B^*)$.
\end{remark}

According to Lemma \ref{lemuniq}, under assumption \eqref{hypnull} (which is implied by the Kalman condition) the static optimal control problem \eqref{staticpbLQ}, whose minimizer is characterized by \eqref{eqxlambda}, has a unique solution $(\bar x,\bar u,\bar\lambda)$.
Setting
$$
\delta x(t) = x_T(t)-\bar x,\quad \delta\lambda(t)=\lambda_T(t)-\bar\lambda,
$$
we get from \eqref{extremalsystLQ} and \eqref{extrLQ}
\begin{equation}\label{eqdelta}
\begin{split}
\delta\dot{x}(t) &= A \delta x(t) + B U^{-1}B^* \delta\lambda(t),  \\
\delta\dot{\lambda}(t) &= Q\delta x(t)-A^* \delta\lambda(t), 
\end{split}
\end{equation}
with $\delta x(0)=x_0-\bar x$ and $\delta x(T)=x_1-\bar x$ (the latter equality being replaced with $\delta\lambda(T)=-\bar\lambda$ in the case where the final point is free).
This is a \textit{shooting problem} (\textit{two-point boundary value problem}) for the linear differential system
\begin{equation}\label{ODEZ}
\dot{Z}(t) = M Z(t) ,
\end{equation}
with
$$
Z(t) = \begin{pmatrix}
\delta x(t)\\ \delta\lambda(t)
\end{pmatrix},
$$
for which a part of the initial data and a part of the final data are imposed, and which consists of determining what is the right initial condition $\delta\lambda(0)$ such that the solution $Z(\cdot)=(\delta x(\cdot),\delta\lambda(\cdot))$ of the differential system \eqref{ODEZ}, starting at 
$$Z(0)=\begin{pmatrix} x_0-\bar x \\ \delta\lambda(0) \end{pmatrix} ,$$
satisfies at the final time the condition $\delta x(T)=x_1-\bar x$ (or $\delta\lambda(T)=-\bar\lambda$ if the final point is free).

The matrix $M$  enjoys the following crucial property, which is at the heart of the proof of the turnpike property.

\begin{lemma}\label{lem_Ham}
The matrix $M$ is Hamiltonian\footnote{This fact in itself implies that there exists a symplectic change of coordinates such that, in the new system, the matrix $M$ consists of blocks either of the form $\begin{pmatrix}
\mu & 0 \\ 0 & -\mu \end{pmatrix}$ with $\mu\in\R$, or $\begin{pmatrix}
0 & \beta \\ -\beta & 0 \end{pmatrix}$ with $\beta\in\R$, or $\begin{pmatrix} S & 0 \\ 0 & -S^* \end{pmatrix}$ with $S=\begin{pmatrix}
\alpha & \beta \\ -\beta & \alpha \end{pmatrix}$ with $(\alpha,\beta)\in\R^2$.
For a more detailed discussion of symplectic normal forms and of their use in control theory, we refer the reader to \cite{BFT}.
Under the additional assumptions that $W$ is positive definite, that $H_{uu}$ is negative definite, and that the pair $(A,B)$ satisfies the Kalman condition, actually in the above decomposition only the first possibility can occur, as shown in the proof of the lemma.}, that is, $M$ belongs to $\mathrm{sp}(n,\R)$, the Lie algebra of the Lie group of symplectic matrices $\mathrm{Sp}(n,\R)$.
If the pair $(A,B)$ satisfies the Kalman condition then all eigenvalues of the matrix $M$ are real and nonzero, and moreover if $\mu$ is an eigenvalue then $-\mu$ is an eigenvalue.
\end{lemma}

\begin{proof}
The proof is borrowed from \cite{WildeKokotovic,AndersonKokotovic}). Let $E_-$ (resp., $E_+$) be the minimal symmetric negative definite matrix (resp., the maximal symmetric positive definite matrix) solution of the algebraic Riccati equation
$$
XA + A^* X + X B U^{-1} B^* X - Q = 0.
$$
Setting 
$$
P = \begin{pmatrix}
I_n & I_n \\ E_- & E_+
\end{pmatrix} ,
$$
the matrix $P$ is invertible and
\begin{equation*}
P^{-1}MP = \begin{pmatrix}
A + BU^{-1}B^* E_- & 0 \\ 0 & A + BU^{-1}B^* E_+
\end{pmatrix}.
\end{equation*}
Moreover, subtracting the Riccati equations satisfied by $E_+$ and $E_-$, we have
$$
(E_+-E_-) ( A + BU^{-1}B^*E_+ ) + ( A + BU^{-1}B^*E_- )^* (E_+-E_-) = 0,
$$
and since the matrix $E_+-E_-$ is invertible it follows that the eigenvalues of $A + BU^{-1}B^*E_+$ are the negative of those of $A + BU^{-1}B^*E_-$, which have negative real parts by a well-known property of the algebraic Riccati theory (see, e.g., \cite{Kwakernaak,Trelatbook}), due to the facts that $(A,B)$ satisfies the Kalman condition, that $W$ and $U$ are positive definite.
\end{proof}

The argument of the proof means that, setting 
$$
Z(t) = \begin{pmatrix}
I_n & I_n \\ E_- & E_+
\end{pmatrix}
Z_1(t)
$$
we get from \eqref{syst_Z} that
\begin{equation}\label{eqZ1LQ}
\dot Z_1(t) = \begin{pmatrix}
A + BU^{-1}B^* E_- & 0 \\ 0 & A + BU^{-1}B^* E_+
\end{pmatrix}
Z_1(t).
\end{equation}
The differential system \eqref{eqZ1LQ} is purely hyperbolic, with the $n$ first equations being the contracting part and the $n$ last ones being the expanding one.
More precisely, setting
$$
Z_1(t) = \begin{pmatrix}
v(t) \\ w(t)
\end{pmatrix},
$$
we have, using \eqref{eqZ1},
\begin{equation*}
\begin{split}
v'(t) &= ( A+BU^{-1}B^*E_- ) v(t), \\
w'(t) &= ( A+BU^{-1}B^*E_+ ) w(t) ,
\end{split}
\end{equation*}
and since all eigenvalues of $A+BU^{-1}B^*E_-$ have negative real parts and since the eigenvalues of $A+BU^{-1}B^*E_+$ are the negative of those of $A+BU^{-1}B^*E_-$, it follows that
\begin{equation}\label{est_expLQ}
\Vert v(t)\Vert \leq \Vert v(0)\Vert e^{-C_2 t},\qquad
\Vert w(t)\Vert \leq \Vert w(T)\Vert e^{-C_2 (T-t)} ,
\end{equation}
for every $t\in[0,T]$, where
$$
C_2 = - \max \{ \Re(\mu) \mid \mu\in\mathrm{Spec}(A+BU^{-1}B^*E_-)  \}   >0 .
$$
This implies that, roughly speaking, one has $v(t)\simeq 0$ and $w(t)\simeq 0$, and therefore $\delta x(t)\simeq 0$ and $\delta\lambda(t)\simeq 0$ as well, for every $t\in[\tau,T-\tau]$ for some $\tau>0$. We are going to be more precise below.
Note that at this step we can see the turnpike property emerge, with a first transient arc, a middle long arc along which $v(t)\simeq 0$ and $w(t)\simeq 0$, and a final transient arc.

\medskip

To finish the proof and get precise estimates,  terminal conditions need to be taken into account. In other words,  we are going to prove  that the above shooting problem is indeed well posed and that the values of $v(0)$ and $w(T)$ can be determined in a univocal way from the terminal conditions. Note that this crucial step is not achieved in \cite{WildeKokotovic,AndersonKokotovic}. The argument is however quite simple in the present case, where the initial point is fixed and the final point is either fixed or free. It will be far more intricate in the general nonlinear case (whence the interest of treating first the present situation, in order to facilitate the readibility).

Since the case where the final point is fixed is similar but slightly simpler than the case where it is free, we only treat the case where $x_T(T)$ is let free, and hence $\lambda_T(T)=0$.
We have then $\delta x(0)=x_0-\bar x$ and $\delta\lambda(T)=-\bar\lambda$, and hence we infer from \eqref{est_expLQ} that
\begin{equation*}
\begin{split}
\Vert v(0) - (x_0-\bar x) \Vert &\leq  \Vert w(T)\Vert e^{-C_2 T}, \\
\Vert w(T) + E_+^{-1}\bar\lambda \Vert &\leq \Vert E_+^{-1}E_-\Vert \Vert v(0)\Vert e^{-C_2 T}, \end{split}
\end{equation*}
and thus
\begin{equation*}
\begin{split}
\Vert v(0) - (x_0-\bar x) \Vert &\leq  \Vert E_+^{-1}\bar\lambda \Vert e^{-C_2 T} +  \Vert E_+^{-1}E_-\Vert \Vert v(0)\Vert e^{-2C_2 T} , \\
\Vert w(T) + E_+^{-1}\bar\lambda \Vert &\leq \Vert E_+^{-1}E_-\Vert \Vert x_0-\bar x\Vert e^{-C_2 T} +  \Vert E_+^{-1}E_-\Vert   \Vert w(T)\Vert e^{-2C_2 T} .
\end{split}
\end{equation*}
This proves that
\begin{equation*}
\begin{split}
v(0) &= x_0-\bar x + \mathrm{O}\left( \Vert E_+^{-1}\bar\lambda \Vert e^{-C_2 T} \right), \\
w(T) &= - E_+^{-1}\bar\lambda + \mathrm{O}\left( \Vert E_+^{-1}E_-\Vert \Vert x_0-\bar x\Vert e^{-C_2 T} \right).
\end{split}
\end{equation*}
At this step, we note that we have determined the values of $v(0)$ and $w(T)$, as announced earlier.
The fact that the shooting method is well posed, and the hyperbolicity feature which implies the turnpike property, are evident on Figure \ref{fig_saddle}.
\begin{figure}[h]
\begin{center}
\resizebox{8cm}{!}{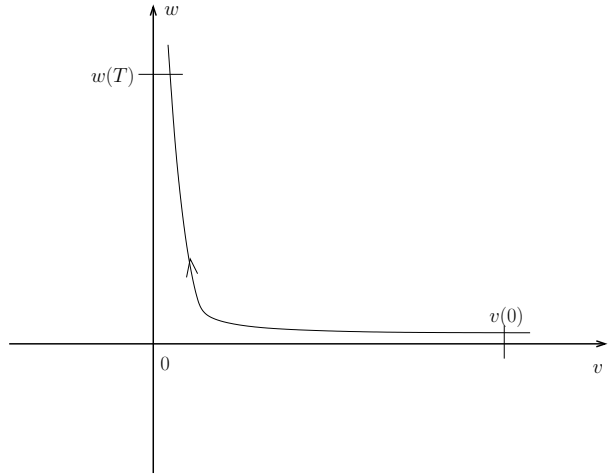}
\end{center}
\caption{Saddle point}\label{fig_saddle}
\end{figure}

Using \eqref{est_expLQ} again, we have the estimates
\begin{equation*}
\begin{split}
\Vert v(t)\Vert &\leq \Vert x_0-\bar x \Vert e^{-C_2 t} + \mathrm{O}\left( \Vert E_+^{-1}\bar\lambda \Vert e^{-C_2 (t+T)} \right),\\
\Vert w(t)\Vert &\leq \Vert E_+^{-1}\bar\lambda\Vert e^{-C_2 (T-t)} + \mathrm{O}\left( \Vert E_+^{-1}E_-\Vert \Vert x_0-\bar x\Vert e^{-C_2 (2T-t) } \right) ,
\end{split}
\end{equation*}
for every $t\in[0,T]$.
Finally, turning back to $\delta x(t)$ and $\delta\lambda(t)$, we have
$\delta x(t) = v(t) + w(t)$ and $\delta\lambda(t) = E_- v(t)+ E_+ w(t)$, and therefore we conclude that
\begin{equation*}
\begin{split}
\Vert \delta x(t)\Vert &\leq \Vert x_0-\bar x \Vert e^{-C_2 t} + \Vert E_+^{-1}\bar\lambda\Vert e^{-C_2 (T-t)}   \\
&\qquad + \mathrm{O}\left( \Vert E_+^{-1}\bar\lambda \Vert e^{-C_2 (t+T)} + \Vert E_+^{-1}E_-\Vert \Vert x_0-\bar x\Vert e^{-C_2 (2T-t)} \right) , \\
\Vert \delta\lambda(t)\Vert &\leq \Vert E_-\Vert \Vert x_0-\bar x \Vert e^{-C_2 t} +\Vert E_+\Vert  \Vert E_+^{-1}\bar\lambda\Vert e^{-C_2 (T-t)} \\
&\qquad + \mathrm{O}\left( \Vert E_-\Vert \Vert E_+^{-1}\bar\lambda \Vert e^{-C_2 (t+T)} + \Vert E_+\Vert  \Vert E_+^{-1}E_-\Vert \Vert x_0-\bar x\Vert e^{-C_2 (2T-t)}  \right)  .
\end{split}
\end{equation*}
The estimate for the control comes from the fact that $\delta u(t) = u_T(t)-\bar u = U^{-1} B^* \delta\lambda(t)$.
The theorem is proved.

\subsection{Proof in the general nonlinear case}\label{sec3.2}
We introduce perturbation variables, by setting
$$
x_T(t)=\bar x+\delta x(t),\quad \lambda_T(t)=\bar\lambda+\delta\lambda(t),\quad u_T(t)=\bar u+\delta u(t).
$$
By linearizing the extremal equations \eqref{extremalsyst} coming from the Pontryagin maximum principle, we easily get
$$
\delta u(t) = -H_{uu}^{-1} \left( H_{xu} \delta x(t) + H_{\lambda u} \delta\lambda(t) \right) + \mathrm{o}(\delta x(t),\delta\lambda(t)),
$$
and then
\begin{equation*}
\begin{split}
\delta \dot x(t) &= \left( H_{x\lambda} - H_{u\lambda}H_{uu}^{-1} H_{xu}\right) \delta x(t) - H_{u\lambda} H_{uu}^{-1} H_{\lambda u} \delta\lambda(t)  + \mathrm{o}(\delta x(t),\delta\lambda(t)), \\
\delta\dot\lambda(t) & = \left( -H_{xx} + H_{ux} H_{uu}^{-1} H_{xu} \right) \delta x(t) + \left( -H_{\lambda x}+H_{ux}H_{uu}^{-1}H_{\lambda u}\right) \delta\lambda(t)  + \mathrm{o}(\delta x(t),\delta\lambda(t)),
\end{split}
\end{equation*}
where the term $\mathrm{o}(\varepsilon)$ stands for terms that can be neglected with respect to the first-order tems $\delta x(t)$, $\delta\lambda(t)$ and $\delta u(t)$.
In other words, setting
$$
Z(t) = \begin{pmatrix}
\delta x(t)\\ \delta\lambda(t)
\end{pmatrix},
$$
we get
\begin{equation}\label{syst_Z}
\dot Z(t) = M Z(t) + \mathrm{o}(Z(t)),
\end{equation}
with
\begin{equation}\label{def_M}
M = \begin{pmatrix}
H_{x\lambda} - H_{u\lambda}H_{uu}^{-1} H_{xu}  &  - H_{u\lambda} H_{uu}^{-1} H_{\lambda u}  \\
-H_{xx} + H_{ux} H_{uu}^{-1} H_{xu}  &  -H_{\lambda x}+H_{ux}H_{uu}^{-1}H_{\lambda u}
\end{pmatrix} =
\begin{pmatrix}
A & -B H_{uu}^{-1} B^* \\
W & -A^*
\end{pmatrix}.
\end{equation}

We stress that all above equations are written at the first order, with remainder terms in $\mathrm{o}(\cdot)$. This is valuable as long as $\Vert\delta x(t)\Vert+\Vert\delta\lambda(t)\Vert+\Vert\delta u(t)\Vert$ remains small. Throughout the forthcoming analysis we make this a priori assumption, which will be indeed satisfied a posteriori as a result of our analysis. 

%
%
%
%

Note that the matrix $M$ has the same form as in the linear quadratic case (see \eqref{defMLQ}), except that the matrix $Q$ is replaced with the matrix $W$.

It is as well a Hamiltonian matrix. Under the assumptions that $W$ is positive definite, that $H_{uu}$ is negative definite, and that the pair $(A,B)$ satisfies the Kalman condition, all eigenvalues of the matrix $M$ are real and nonzero, and moreover if $\lambda$ is an eigenvalue then $-\lambda$ is an eigenvalue.

The proof of this fact is the same as in Lemma \ref{lem_Ham}: we define $E_-$ (resp., $E_+$) as the minimal symmetric negative definite matrix (resp., the maximal symmetric positive definite matrix) solution of the algebraic Riccati equation
$$
XA + A^* X - X B H_{uu}^{-1} B^* X - W = 0.
$$
Then, setting 
$$
P = \begin{pmatrix}
I_n & I_n \\ E_- & E_+
\end{pmatrix} ,
$$
the matrix $P$ is invertible and
\begin{equation}\label{CDV}
P^{-1}MP = \begin{pmatrix}
A - BH_{uu}^{-1}B^* E_- & 0 \\ 0 & A - BH_{uu}^{-1}B^* E_+
\end{pmatrix}.
\end{equation}
Moreover, subtracting the Riccati equations satisfied by $E_+$ and $E_-$, we have
$$
(E_+-E_-) ( A - BH_{uu}^{-1}B^*E_+ ) + ( A - BH_{uu}^{-1}B^*E_- )^* (E_+-E_-) = 0,
$$
and since the matrix $E_+-E_-$ is invertible it follows that the eigenvalues of $A - BH_{uu}^{-1}B^*E_+$ are the negative of those of $A - BH_{uu}^{-1}B^*E_-$, which have negative real parts (as  stated by the algebraic Riccati theory).
Now, setting 
$$
Z(t) = \begin{pmatrix}
I_n & I_n \\ E_- & E_+
\end{pmatrix}
Z_1(t)  ,
$$
we get from \eqref{syst_Z} that
\begin{equation}\label{eqZ1}
\dot Z_1(t) = \begin{pmatrix}
A - BH_{uu}^{-1}B^* E_- & 0 \\ 0 & A - BH_{uu}^{-1}B^* E_+
\end{pmatrix}
Z_1(t) + \mathrm{o}(Z_1(t)).
\end{equation}
The differential system \eqref{eqZ1} is purely hyperbolic, with the $n$ first equations being the contracting part and the $n$ last ones being the expanding one.
More precisely, setting
$$
Z_1(t) = \begin{pmatrix}
v(t) \\ w(t)
\end{pmatrix},
$$
we have, using \eqref{eqZ1},
\begin{equation*}
\begin{split}
v'(t) &= ( A-BH_{uu}^{-1}B^*E_- ) v(t) + \mathrm{o}(v(t),w(t)) , \\
w'(t) &= ( A-BH_{uu}^{-1}B^*E_+ ) w(t) + \mathrm{o}(v(t),w(t)) ,
\end{split}
\end{equation*}
and since all eigenvalues of $A-BH_{uu}^{-1}B^*E_-$ have negative real parts and since the eigenvalues of $A-BH_{uu}^{-1}B^*E_+$ are the negative of those of $A-BH_{uu}^{-1}B^*E_-$, it follows that
\begin{equation}\label{est_exp}
\begin{split}
\Vert v(t)\Vert &\leq \Vert v(0)\Vert e^{-\frac{C_2}{2} t} + \Vert w(T)\Vert\,\mathrm{o}(e^{-C_2(T-t)}),\\
\Vert w(t)\Vert &\leq \Vert w(T)\Vert e^{-\frac{C_2}{2} (T-t)}  + \Vert v(0)\Vert\,\mathrm{o}(e^{-C_2 t}),
\end{split}
\end{equation}
for every $t\in[0,T]$, where
$$
C_2 = - \max \{ \Re(\mu) \mid \mu\in\mathrm{Spec}(A-BH_{uu}^{-1}B^*E_-)  \}   >0 .
$$

\medskip

The next step of the proof consists of taking into account the general terminal conditions \eqref{terminalcond} and the corresponding transversality conditions \eqref{condtransv}, and to prove that the shooting problem is indeed well posed under the assumptions made in the statement of the theorem.
Due to the generality of our terminal conditions, this part of the proof is far more technical than in the previous linear quadratic case where the initial point was fixed and the final point was either fixed or free.

Let us linearize also the terminal conditions \eqref{terminalcond} and the corresponding transversality conditions \eqref{condtransv}. Since $x_T(0)=\bar x+\delta x(0)$ and $x_T(T)=\bar x+\delta x(T)$, we get from \eqref{terminalcond} that
\begin{equation}\label{tc1}
R_x \delta x(0) + R_y\delta x(T) = -R(\bar x,\bar x) + \mathrm{o}(\delta x(0),\delta x(T)),
\end{equation}
where
$$
R_x = \frac{\partial R}{\partial x}(\bar x,\bar x) \quad\textrm{and}\quad
R_y = \frac{\partial R}{\partial y}(\bar x,\bar x)
$$
are matrices of size $k\times n$.
Similarly, since $\lambda_T(0)=\bar\lambda+\delta\lambda(0)$ and $\lambda_T(T)=\bar\lambda+\delta\lambda(T)$, we get from \eqref{condtransv} that
\begin{equation}\label{tc2}
\begin{split}
-\bar\lambda-\delta\lambda(0) &= \sum_{i=1}^k \gamma_i \left(  
\nabla_x R^i(\bar x,\bar x) + \frac{\partial^2 R^i}{\partial x^2}(\bar x,\bar x) \delta x(0) + \frac{\partial^2 R^i}{\partial x\partial y}(\bar x,\bar x) \delta x(T) \right) + \mathrm{o}(\delta x(0),\delta x(T))  , \\
\bar\lambda+\delta\lambda(T) &= \sum_{i=1}^k \gamma_i \left(  
\nabla_y R^i(\bar x,\bar x) + \frac{\partial^2 R^i}{\partial y\partial x}(\bar x,\bar x) \delta x(0) + \frac{\partial^2 R^i}{\partial y^2}(\bar x,\bar x) \delta x(T) \right) + \mathrm{o}(\delta x(0),\delta x(T))  .
\end{split}
\end{equation}
Note that, under our a priori assumption, \eqref{tc1} implies that $R(\bar x,\bar x)=\mathrm{O}(\delta x(0),\delta x(T))$, and that \eqref{tc2} implies that 
\begin{equation}\label{12:03}
\left\Vert \begin{pmatrix} -\bar\lambda\\ \bar\lambda\end{pmatrix} - \sum_{i=1}^k \gamma_i \nabla R^i(\bar x,\bar x) \right\Vert = \mathrm{O}(\delta x(0),\delta x(T),\delta\lambda(0),\delta\lambda(T)).
\end{equation}
This will be possible thanks to the assumption \eqref{defect} on the smallness of $\bar D$.

In what follows, we set
$$
\Gamma = \begin{pmatrix} \gamma_1 \\ \vdots \\ \gamma_k \end{pmatrix}.
$$
The system of equations \eqref{tc1}--\eqref{tc2} is a system of $2n+k$ equations in the $2n+k$ unknowns $(v(0),w(T),\Gamma)$. We are going to prove that this system (which exactly represents the shooting problem) is well posed.

\begin{lemma}\label{lembeforeQ}
There exists
$$
\bar\Gamma = \begin{pmatrix} \bar\gamma_1 \\ \vdots \\ \bar\gamma_k \end{pmatrix}
$$
such that
$$
\begin{pmatrix} -\bar\lambda\\ \bar\lambda\end{pmatrix} - \sum_{i=1}^k \gamma_i \nabla R^i(\bar x,\bar x)  = \begin{pmatrix} -\bar\lambda - R_x^*\bar\Gamma \\ \bar\lambda - R_y^* \bar\Gamma \end{pmatrix} =\begin{pmatrix} 0 \\ 0\end{pmatrix}.
$$
\end{lemma}

\begin{proof}
By assumption, the point $(\bar x,\bar x)$ is not a singular point of $R$, and therefore the differential $dR(\bar x,\bar x)=\begin{pmatrix} R_x & R_y\end{pmatrix}$ (which is a matrix of size $k\times 2n$) is of maximal rank, that is $k$. Then the matrix $R_xR_x^*+R_yR_y^*$ is invertible, and
$$
\bar\Gamma = (R_xR_x^*+R_yR_y^*)^{-1} (-R_x+R_y)\bar\lambda
$$
makes the job.
\end{proof}

We now set
$$
\Gamma = \begin{pmatrix} \gamma_1 \\ \vdots \\ \gamma_k \end{pmatrix}
= \bar\Gamma + \delta\Gamma = \begin{pmatrix} \bar\gamma_1 \\ \vdots \\ \bar\gamma_k \end{pmatrix} + = \begin{pmatrix} \delta\gamma_1 \\ \vdots \\ \delta\gamma_k \end{pmatrix}.
$$
It follows from \eqref{12:03} that
$$
\delta\Gamma = \mathrm{O}(\delta x(0),\delta x(T),\delta\lambda(0),\delta\lambda(T)).
$$
Now, from \eqref{est_exp}, \eqref{tc1} and \eqref{tc2}, we infer that
\begin{equation}\label{systshooting}
\begin{pmatrix}
R_x & R_y & 0 \\
E_-+N_1 & N_2 & R_x^* \\
N_3 & -E_+N_4 & R_y^*
\end{pmatrix}
\begin{pmatrix}
v(0) \\ w(T) \\ \delta\Gamma
\end{pmatrix}
= \begin{pmatrix} -R(\bar x,\bar x)\\ 0\\ 0\end{pmatrix}
+ \mathrm{O}(e^{-\frac{C_2}{2}T}) + \mathrm{o}(v(0),w(T),\delta\Gamma),
\end{equation}
with
$$
N_1 = \sum_{i=1}^k \bar\gamma_i  \frac{\partial^2 R^i}{\partial x^2}(\bar x,\bar x) , \qquad
N_2 = \sum_{i=1}^k \bar\gamma_i \frac{\partial^2 R^i}{\partial y\partial x}(\bar x,\bar x), 
$$
$$
N_3 = \sum_{i=1}^k \bar\gamma_i \frac{\partial^2 R^i}{\partial y\partial x}(\bar x,\bar x), \qquad
N_4 = \sum_{i=1}^k \frac{\partial^2 R^i}{\partial y^2}(\bar x,\bar x) .
$$
This linear system of $2n+k$ equations in the $2n+k$ unknowns $(v(0),w(T),\delta\Gamma)$ represents the above-mentioned shooting problem.
We can prove two facts concerning the invertibility of the matrix
\begin{equation}\label{defQ}
Q = \begin{pmatrix}
R_x & R_y & 0 \\
E_-+N_1 & N_2 & R_x^* \\
N_3 & -E_++N_4 & R_y^*
\end{pmatrix}
\end{equation}
of this system.

\begin{lemma}\label{lemQ}
Consider the mapping $R:\R^n\times\R^n\rightarrow\R^k$ standing for the terminal conditions.
\begin{enumerate}
\item If the norm of the Hessian of $R$ at $(\bar x,\bar x)$ is small enough (this is the case if the terminal conditions are almost linear) then the matrix $Q$ is invertible.
\item We endow the set $\mathcal{X}$ of mappings $R:\R^n\times\R^n\rightarrow\R^k$ with the $C^2$ topology. There exists a stratified (in the sense of Whitney) submanifold $\mathcal{S}$ of $\mathcal{X}$ of codimension greater than or equal to one such that, for every $R\in\mathcal{X}\setminus\mathcal{S}$, the corresponding matrix $Q$ is invertible.
\end{enumerate}
\end{lemma}

This lemma means that for generic terminal conditions, the matrix $Q$ is invertible. Note that the set $\mathcal{X}\setminus\mathcal{S}$ is open, dense and contains a neighborhood of the subset of linear mappings.

\begin{proof}
Let us prove the first point. If the Hessian of $R$ is almost equal to zero, then this means that the matrices $N_1$, $N_2$, $N_3$ and $N_4$ have a small norm. In particular this implies that the matrix $Q$ is close to the matrix
$$
Q_0 = \begin{pmatrix}
R_x & R_y & 0 \\
E_- & 0 & R_x^* \\
0 & -E_+ & R_y^*
\end{pmatrix}.
$$
Let us prove that $Q_0$ is invertible (and hence $Q$ is invertible if the norm of the Hessian of $R$ is small enough). To prove this fact, let us solve the system
\begin{equation*}
\begin{split}
R_x X_1 + R_y X_2 &= Y_1, \\
E_-X_1 + R_x^*X_3 &= Y_2,\\
-E_+X_2+R_y^*X_3 &= Y_3.
\end{split}
\end{equation*}
From the second and third equations, we infer that $X_1=-E_-^{-1} R_x^* X_3 -E_-^{-1}Y_2$ and that $X_2 = E_+^{-1} R_y^* X_3 - E_+^{-1} Y_3$, and plugging into the first equation yields that
$$
(-R_xE_-^{-1} R_x^* + R_y E_+^{-1} R_y^*) X_3 = Y_1 + R_xE_-^{-1}Y_2 + R_y E_+^{-1} Y_3.
$$
This equation can be solved because the matrix
$$
-R_xE_-^{-1} R_x^* + R_y E_+^{-1} R_y^* = R_x (-E_-)^{-1/2} (R_x (-E_-)^{-1/2})^* + R_y E_+^{-1/2} (R_y E_+^{-1/2})^* 
$$
is invertible. This comes again from the fact that the matrix $\begin{pmatrix} R_x & R_y\end{pmatrix}$ has maximal rank $k$. The first point of the lemma follows.

The second point of the lemma easily follows from the fact that the relation $\mathrm{det}\, Q=0$ is an analytic equation in the coefficients of the differential of $R$ and of the Hessian of $R$ at the point $(\bar x,\bar x)$. By the well-known subanalyticity theory, this analytic set is a stratified submanifold of $\mathcal{X}$ of codimension greater than or equal to one.
\end{proof}

Under the conditions of this lemma, $Q$ is invertible and therefore the system \eqref{systshooting} is well posed and has a unique solution $(v(0),w(T),\delta\Gamma)$. Moreover, our analysis shows that, under the assumption \eqref{defect}, our a priori assumption indeed holds true and the norm of $(v(0),w(T),\delta\Gamma)$ is of the order of $\varepsilon$, for $T>0$ large enough.

The end of the proof is then similar to the proof done in the previous section.
Indeed, at this step we have proved that the values of $v(0)$ and $w(T)$ are determined in a univocal way. The hyperbolicity of the system \eqref{eqZ1}, represented on Figure \ref{fig_saddle}, implies as well the desired turnpike property, in the form of the estimate \eqref{estimate_turnpike}.

\medskip

\paragraph{Acknowledgment.}
This work was achieved while the second author was visiting the Laboratoire Jacques-Louis Lions with the support of the Paris City Hall ``Research in Paris" program. E. Zuazua was also partially supported by Grants MTM2008-03541 and MTM2011-29306 of MICINN
Spain, Project PI2010-04 of the Basque Government, ERC Advanced
Grant FP7-246775 NUMERIWAVES and ESF Research Networking Programme
OPTPDE.


\end{document}